\documentclass[12pt]{article}

\usepackage{amsmath,amsfonts,amssymb,amsthm,mathrsfs}
\usepackage{xcolor}
\usepackage{color} 

\theoremstyle{plain}
\newtheorem{theorem}{Theorem}[section]

\newtheorem{lemma}[theorem]{Lemma}
\newtheorem{proposition}[theorem]{Proposition}

\newtheorem{example}[theorem]{Example}

\theoremstyle{definition}
\newtheorem{definition}[theorem]{Definition}
\newtheorem{remark}[theorem]{Remark}

\newcommand{\mybibitem}[7]{\bibitem[#1]{#1} #2. {\it #3,} #4 {\bf #5}
#6 #7.}
\newcommand{\mybibsix}[6]{\bibitem[#1]{#1} #2. {\it #3,} #4 {\bf #5}
#6.}

\newcommand{\N}{{\mathbb N}}
\newcommand{\R}{{\mathbb R}}
\newcommand{\Rn}{{\R^n}}
\newcommand{\Z}{{\mathbb Z}}

\newcommand{\al}{\alpha}
\newcommand{\eps}{\varepsilon}

\newcommand{\teta}{\vartheta}
\newcommand{\vrho}{\varrho}
\newcommand{\pip}{\varphi}
\newcommand{\Om}{\Omega}

\newcommand{\cE}{\mathcal{E}}
\newcommand{\dmu}{d\mu}

\DeclareMathOperator{\Capa}{Cap}

\newcommand{\scal}[2]{\langle #1,#2 \rangle}
\DeclareMathOperator{\Lip}{Lip}

\newcommand{\DMd}{\mathcal{DM}^2}
\newcommand{\DMi}{\mathcal{DM}^\infty}

\def\vint_#1{\mathchoice%
  {\mathop{\kern 0.2em\vrule width 0.6em height 0.69678ex depth
      -0.58065ex \kern -0.8em \intop}\nolimits_{\kern -0.4em#1}}%
  {\mathop{\kern 0.1em\vrule width 0.5em height 0.69678ex depth
      -0.60387ex \kern -0.6em \intop}\nolimits_{#1}}%
  {\mathop{\kern 0.1em\vrule width 0.5em height 0.69678ex depth
      -0.60387ex \kern -0.6em \intop}\nolimits_{#1}}%
  {\mathop{\kern 0.1em\vrule width 0.5em height 0.69678ex depth
      -0.60387ex \kern -0.6em \intop}\nolimits_{#1}}}

\newcommand{\diver}{\mathop {\rm div}}

\newcommand\res{\mathop{\mbox{\vrule height 7pt width .5pt depth 0pt
      \vrule height .5pt width 6pt depth 0pt}}\nolimits}

\begin{document}

\title{Boundary measures, generalized Gauss--Green formulas, and mean
  value property in metric measure spaces} \author{ Niko Marola,
  Michele Miranda Jr., and Nageswari Shanmugalingam\footnote{N. S. was
    partially supported by NSF grant DMS-0243355, Simons Foundation
    grant~\#200474, and NSF grant DMS-1200915.}}

\date{}
\maketitle

\tableofcontents

\begin{abstract}
  We study mean value properties of harmonic functions in metric
  measure spaces. The metric measure spaces we consider have a
  doubling measure and support a $(1,1)$-Poincar\'e inequality. The
  notion of harmonicity is based on the Dirichlet form defined in
  terms of a Cheeger differentiable structure. By studying fine
  properties of the Green function on balls, we characterize harmonic
  functions in terms of a mean value property. As a consequence, we
  obtain a detailed description of Poisson kernels. We shall also
  obtain a Gauss--Green type formula for sets of finite perimeter
  which posses a Minkowski content characterization of the
  perimeter. For the Gauss--Green formula we introduce a suitable
  notion of the interior normal trace of a regular ball.
\end{abstract}


\medskip

\noindent {\it Keywords\,}: Dirichlet form, doubling measure,
functions of bounded variation, Gauss--Green theorem, Green function,
harmonic function, metric space, Minkowski content, Newtonian space,
perimeter measure, Poincar\'e inequality, singular function.

\section{Introduction}

Solving the Dirichlet problem on a smooth domain in $\R^n$ is
equivalent to constructing harmonic measure on the boundary of the
domain. More precisely, it is known that the classical harmonic
measure can be expressed in terms of a Poisson kernel which is given
by the Radon--Nikodym derivative of harmonic measure with respect to
the Hausdorff boundary measure; that is
\[
P(x,y)=\frac{d\nu_{x}}{d\mathcal{H}^{n-1}}(y).
\]

In general metric measure spaces with a doubling measure and a
Poincar\'e inequality the Dirichlet problem has been solved for
Sobolev type boundary data in \cite{Sh2}, and also for all continuous
boundary values in \cite{BBS}. In \cite{BBS} the authors provide an
integral representation for the solution to the Dirichlet problem, and
hence extend the solvability to $L^1$ boundary data. In this general
setting, however, due to lack of a natural choice of boundary
Hausdorff measure one has to replace a Poisson kernel with a Poisson
kernel-like object for which
\[
P_{x_0}(x,y)=\frac{d\nu_{x}}{d\nu_{x_0}}(y).
\]
It was shown in \cite{BBS} that for a fixed
$x_0\in \Om$, where $\Om$ is a bounded open subset of $X$, there exists
a Radon measure $\nu_{x_0}$ concentrated on $\partial \Om$,
i.e. $\nu_{x_0}$ is a harmonic measure on $\partial\Om$ evaluated at
$x_0$, and a real-valued function $P_{x_0}$ on $\Om\times\partial\Om$
such that whenever $f\in L^1(\partial\Om,\nu_{x_0})$ the following
expression for the harmonic extension $H_f$ is valid:
\[
H_f(x)=\int_{\partial \Om}f(y)P_{x_0}(x,y)\, d\nu_{x_0}(y),
\]
and moreover, for each $y\in\partial\Om$ the function
$P_{x_0}(\cdot,y)$ is harmonic in $\Om$.
  
Our main objective is to find a relationship between the Poisson
kernel that generates solutions to the Dirichlet problem in terms of
Cheeger differentiable structure, and the perimeter measure of a ball
of finite perimeter in metric measure spaces. Our framework is a
complete geodesic metric measure space with a doubling Borel measure, and we
moreover assume that the space supports a $(1,1)$-Poincar\'e inequality. These
conditions are discussed in detail in Section~\ref{sect:Prel}. We
shall describe the Poisson kernel in terms of an analog of a normal
derivative of the Green function at the boundary.

We also study divergence-measure fields along the lines of
Ziemer~\cite{Z} in this general context. We consider an $L^2$-vector
field, $\vec F$, from a metric measure space $X$ to $\R^k$ for which
$\diver \vec F$ is a real-valued signed Borel measure with finite mass. 

To investigate divergence-measure fields we shall provide a
meaningful definition for the divergence operator in metric measure
spaces. We then generalize some results obtained in \cite{Z} to the
metric setting. In particular, we obtain the Gauss--Green type
integration by parts formula for sets of finite perimeter which
possess a Minkowski content characterization of the perimeter. For the
Gauss--Green formula we introduce a suitable notion of the interior
normal trace of a regular ball.

We mention a related paper by Thompson and Thompson~\cite{TT} in which
the authors define divergence and prove an analogue of the
Gauss--Green theorem in Minkowski spaces, i.e. in finite-dimensional
real normed spaces with smooth and strictly convex unit ball.

We use the results for the divergence operator to characterize the Laplace operator
of the Green function on regular balls as the sum of 
the Dirac point mass and a measure concentrated on the boundary of
the ball. 
This characterization allows us to give a precise description of the Poisson kernel defined in \cite{BBS}. In the setting of Heisenberg groups, we explain the relation between this measure and the perimeter measure or the codimension one Hausdorff measure.

\section{Preliminaries}
\label{sect:Prel}

Here we recall some basic definitions and the notation we shall use in this
paper. Our framework is given by a complete metric measure space
$(X,d,\mu)$, where $\mu$ is {\it doubling}, that is, 
there is a constant $c>0$ such that for every ball $B=B_r(x)$, $x\in X$
and $r>0$,
\begin{equation}\label{doubling}
0<\mu(2B)\leq c\, \mu(B)<\infty.
\end{equation}
We write $B_r(x)$ for the ball centered at $x$ with
radius $r>0$, and
$\lambda B=B_{\lambda r}(x)$ for any $\lambda>0$. The
smallest value of $c$ for which (\ref{doubling}) is valid is called
the doubling constant of $X$, and we shall denote it as $c_d$.

An upper gradient for an extended real-valued function $u:X \to
[-\infty,+\infty]$ is a Borel function $g:X\to[0,\infty]$ such that
\begin{equation}\label{upperGradient}
|u(\gamma(0))-u(\gamma(l_\gamma))|\leq \int_\gamma g\, ds
\end{equation}
for every nonconstant compact rectifiable curve $\gamma:[0,l_\gamma]\to X$.
We say that $g$ is a $p$--weak upper gradient of $u$ if
\eqref{upperGradient} holds for $p$--almost every curve; the notion of
$p$--almost every curve is in the sense of the $p$--modulus of a curve
family $\Gamma$ defined as
\[ {\rm Mod}_p(\Gamma)=\inf \left\{\int_X \varrho^p\,d\mu: \varrho\geq
  0 \textrm{ is a Borel function, } \int_\gamma \varrho\, ds \geq 1
  \textrm{ for all } \gamma\in \Gamma\right\}.
\]
If $u$ has an upper gradient in $L^p(X,\mu)$, then it is possible to
prove the existence of a unique minimal $p$--weak upper gradient
$g_u\in L^p(X,\mu)$ of $u$, where $g_u\leq g$ $\mu$-a.e. for every
$p$--weak upper gradient $g\in L^p(X,\mu)$ of $u$. We refer to
\cite{Sh2} for the case $p>1$, and for the case $p=1$ to \cite{Haj03}.

In what follows, the metric space is supposed to support a weak
$(1,1)$--Poincar\'e inequality: there exist constants $c>0$ and
$\lambda\geq 1$ such that for all balls $B_r$ with $B_{\lambda
  r}\subset X$, for any Lipschitz function $f\in \Lip(X)$ and minimal
$p$--weak upper gradient $g_f$ of $f$ we have
\begin{equation} \label{eq:PI}
\vint_{B_r}|f-f_{B_r}|\,d\mu \leq c r \vint_{B_{\lambda r}}
g_f\,d\mu,
\end{equation}
where 
\[
f_{B_r}:= \vint_{B_r}f\, d\mu :=\frac{1}{\mu(B_r)}\int_{B_r} f\,d\mu
\] 
is the integral average of $f$ on $B_r(x)$. 

It is well known that the doubling condition and the Poincar\'e
inequality imply the quasiconvexity of the metric space $X$, see
\cite{K} and \cite{HaKo}. Therefore, up to a bi--Lipschitz change
of the metric, the space $X$ can be assumed to be geodesic, that is,
given $x,y\in X$ there is a curve $\gamma$ with end points $x$, $y$
and length $d(x,y)$. Moreover, for a geodesic space the weak
$(1,1)$--Poincar\'e inequality implies the $(1,1)$--Poincar\'e
inequality, i.e. \eqref{eq:PI} holds with $\lambda=1$. Therefore, as most of
the properties of metric spaces we consider are bi--Lipschitz
invariant, it is not restrictive to assume that {\it $X$ is a geodesic
  space and supports a $(1,1)$--Poincar\'e inequality}.

We remark here that up to Proposition~\ref{prop:TotalVarMeasure}
  assuming only a $(1,2)$--Poincar\'e inequality would
  suffice. However, in Proposition~\ref{prop:TotalVarMeasure}, and
  what follows thereafter, a $(1,1)$--Poincar\'e inequality is needed,
  for instance, to conclude that the minimal 1-weak upper gradient is
  equal $\mu$-a.e. to its pointwise Lipschitz-constant function.

As proved by Cheeger in \cite{C}, in our setting the following
differentiable structure is given. There exists a countable measurable
covering $U_\al$ of $X$, and Lipschitz coordinate charts
$X^\al=(X_1^\al,\ldots,X_{k_\al}^\al):X\to\R^{k_\al}$ such that
$\mu(U_\al)>0$ for each $\al$, $\mu(X\setminus \bigcup_\al U_\al)=0$
and for all $\al$ the following holds: the charts
$(X_1^\al,\ldots,X_{k_\al}^\al)$ are linearly independent on $U_\al$
and $1\leq k_\al\leq N$, where $N$ is a constant depending on the
doubling constant and the constants from the $(1,1)$-Poincar\'e
inequality satisfying the following condition: For any Lipschitz
function $f:X\to\R$ there is an associated unique (up to a set of zero
$\mu$-measure) measurable function $d_\al f:U_\al\to \R^{k_\al}$ for
which the following Taylor-type approximation
\begin{equation}\label{taylorExp}
f(x)=f(x_0)+d_\al f(x_0) \cdot (X_\al(x)-X_\al(x_0))+o(d(x,x_0))
\end{equation}
holds for $\mu$-a.e. $x_0\in U_\al$. 

The previous construction implies, in particular, that for $x\in
U_\al$ there exists a norm $\|\cdot\|_x$ on $\R^{k_\al}$ equivalent to
the Euclidean norm $|\cdot|$, such that $g_f(x)=\|d_\al f(x)\|_x$ for 
almost every $x\in U_\al$. 
Moreover, it is possible to show that there exists a constant $c>1$
such that
\[
\frac{1}{c}g_f(x) \leq |df(x)|\leq cg_f(x)
\]
for all Lipschitz functions $f$ and $\mu$-a.e. $x\in X$. By $df(x)$ we
mean $d_\al f(x)$ whenever $x\in U_\al$. Indeed, one can choose the
cover such that $U_\al\cap U_\beta$ is empty whenever $\al\ne\beta$.

Formula \eqref{taylorExp} implies in particular linearity of the
operator $f\mapsto df$ and also the Leibniz rule $d(fg)=fdg+gdf$ holds
for all Lipschitz functions $f$ and $g$. 

For the definition of the Sobolev spaces $N^{1,p}(X,\mu)$ we will
follow \cite{Sh1}. Since we assume $X$ to satisfy the $(1,1)$-Poincar\'e inequality, the 
Sobolev space $N^{1,p}(X,\mu)$, $1\leq p<\infty$, can also be defined as the closure of the collection of Lipschitz functions on $X$ in the following $N^{1,p}$-norm
\[
\|u\|_{1,p}^p=\|u\|_{L^p(X)}^p+\|g_u\|_{L^p(X)}^p.
\]
The space $N^{1,p}(X,\mu)$ equipped with the
$N^{1,p}$-norm is a Banach space and a lattice \cite{Sh1}.

Let $E\subset X$ be a Borel set. The $p$--capacity of $E$ is defined
as usual to be the number
\[
\Capa_p(E) = \inf_u\left(\int_X|u|^p\, d\mu + \int_X|du|^p\,
  d\mu\right),
\]
where the infimum is taken over all $u\in N^{1,p}(X,\mu)$ for which
$u=1$ on $E$. We say that a property holds $p$--quasieverywhere,
$p$--q.e. for short, if the set of points for which the property does
not hold has $p$--capacity zero. For instance, if $u,v\in
N^{1,p}(X,\mu)$ and $u=v$ $\mu$-a.e., then $u=v$ $p$--q.e. and
$\|u-v\|_{1,p}=0$. If we, moreover, redefine a function $u\in
N^{1,p}(X,\mu)$ on a set of zero $p$--capacity, then it remains a
representative of the same equivalence class in $N^{1,p}(X,\mu)$.

We shall also use Sobolev spaces defined on a domain $\Om$ (i.e. a
non-empty open pathconnected set) of $X$; the space $N^{1,2}(\Om,\mu)$
is defined in the same way the space $N^{1,2}(X,\mu)$ is, but
considering $\Om$ as the ambient space. The space of Sobolev
functions with zero boundary values is instead defined as
\[
 N^{1,p}_0(\Om,\mu)=\left\{
u\in N^{1,p}(X,\mu): u=0\ p\textrm{-q.e. on } X\setminus \Om
\right\}.
\]
We have that $N^{1,p}_0(\Om,\mu)=N^{1,p}(X,\mu)$ as Banach spaces if
and only if $\text{Cap}_p(X\setminus \Om)=0$. 

In what follows, let $p=2$. By \cite{FHK}, the Cheeger differentiable
structure extends to all functions in $N^{1,2}(X,\mu)$ and
$N^{1,2}(\Om,\mu)$, and hence we define an inner product on
$N^{1,2}(X,\mu)$ by the Dirichlet form
\[
\cE(u,v)=\int_X \scal{du}{dv}\,d\mu, 
\]
for all $u,v \in N^{1,2}(X,\mu)$.  It can be proved that such a form
is strongly regular with the domain, or core, given by $N^{1,2}(X,\mu)$. 

We recall that a Dirichlet form is said to be \emph{strongly regular}
if there exists a subset $K$ of the domain of the Dirichlet form,
dense in both this domain and in the class of Lipschitz functions on
$X$, such that the distance $d_{\cE}:X\times X\to[0,\infty]$ defined,
in our case, by
\[
   d_{\cE}(x,y)=\sup\left\{\pip(x)-\pip(y):\ |d\pip(x)|\le 1\right\}
\]
is a metric on $X$ that induces the same topology on $X$ as the
original metric topology on $X$. In fact, under the doubling property
and a Poincar\'e inequality $d_{\cE}$ is bi-Lipschitz equivalent to
the original metric $d$ on $X$, and so the Dirichlet form $\cE(u,v)$
is strongly regular. The set $K$ is called a core of $\cE$.  We refer
to \cite{Sturm} and \cite{FOT} for more details.

For each $\alpha>0$ we define the
bilinear form
\[
\cE_\al(u,v)=\al \int_X uv\,d\mu  +\cE(u,v).
\]
We thus have on $N^{1,2}(X,\mu)$ the norm $\|\cdot\|_\al$ induced by
$\cE_\al$ which is equivalent to the $N^{1,2}$-norm. In this way,
$N^{1,2}(X,\mu)$ with the norm $\|\cdot\|_\al$ is a Hilbert space with
inner product $\cE_\al$. Note that $\cE$ by itself is not an inner
product on $N^{1,2}(X,\mu)$; $\cE(u,u)=0$ if and only if $u$ is a
constant (see \cite{C}). If, for example, $\mu(X)<\infty$, then
$\cE(u,u)=0$ does not imply that $u=0$. 

The fact that the bilinear form $\cE_\al$ yields a Hilbert space can be
seen as follows. Since the $N^{1,2}$-norm is comparable to the
$\cE_\al$-norm, we have that $N^{1,2}(X,\mu)$ is complete also with
respect to the $\cE_\al$-norm. In this way
the $\cE_\al$-norm is well defined for any $u\in N^{1,2}(X,\mu)$. By
approximation and the linearity of the map $u\mapsto du$, the Leibniz
rule follows for functions $u$ and $v$ in $N^{1,2}(X,\mu)$ (we refer
for these properties to the paper \cite{FHK}).  

\begin{remark}
  We point out that the convergence of a sequence $(u_k)_k$ to a
  function $u$ in $N^{1,2}(X,\mu)$ is same as the convergence of the
  two sequences $(u_k-u)_k$ and $(g_{u_k-u})_k$ to $0$ in
  $L^2(X,\mu)$.

  In general, the convergence of $u_k$ to $u$
  in $L^2(X,\mu)$ together with the convergence of $g_{u_k}$ to $g_u$
  in $L^2(X,\mu)$ does not imply that $u_k$ converges to $u$ in
  $N^{1,2}(X,\mu)$. As a counterexample, consider the metric space
  $X=\R^2$ with the distance induced by the norm $\|(x,y)\|_1=|x|+|y|$
  and with $\mu$ the Lebesgue measure; in this case the upper gradient
  is determined by the dual norm
  $\|(x,y)\|_\infty=\max\{|x|,|y|\}$. It suffices to verify this for a
  Lipschitz function $u$. For such function, by \cite{C}, denoting by
  $B^{(1)}_r(x_0,y_0)$ the ball in the norm $\|\cdot\|_1$ with radius
  $r$ centered at $(x_0,y_0)$, we have that
\begin{align*}
g_u(x_0,y_0) & = \lim_{r\to 0} \sup_{(x,y)\in B^{(1)}_r(x_0,y_0)}
\frac{|u(x,y)-u(x_0,y_0)|}{r} \\
& =\max_{\|v\|_1=1} \scal{\nabla
  u(x_0,y_0)}{v}_{\R^2}=\|\nabla u(x_0,y_0)\|_\infty.
\end{align*}
The sequence $u_k(x,y)=x+f_k(y)$, where $f_k(y)={\rm
  dist}(y,\frac{1}{k}\Z)$ converges to the function $u(x,y)=x$, but
for a.e. point
\[
g_{u_k}(x,y)=\|\nabla u_k(x,y)\|_\infty=1=\|\nabla u(x,y)\|_\infty=g_u(x,y) 
\]
and
\[
g_{u_k-u}(x,y)=\|\nabla u_k(x,y)-\nabla u(x,y)\|_\infty=\|\nabla f_k(y)\|_\infty=1.
\]
Nevertheless, it is possible to use Mazur's lemma to prove that for a
convex combination the aforementioned property holds true, both for
the Cheeger differentiable structure and for the upper gradient.   
For the Cheeger differentiable structure,
however, it is not necessary to take convex combinations. Indeed, in
this case the sequence of gradients $du_h$ is bounded in
$L^2(X,\R^k,\mu)$, and so it is weakly convergent to some $\varphi\in
L^2(X,\R^k,\mu)$.  Mazur's lemma is then needed only to show that
$\varphi=du$. We can consider convex combinations
\[
v_k=\sum_{i=1}^{N(k)} \lambda^{(k)}_i u_i 
\]
with strong convergence $v_k\to u$ in $L^2(X,\mu)$ and $dv_k\to
\varphi$ in $L^2(X,\R^k,\mu)$, that is $v_k\to u$ in $N^{1,2}(X,\mu)$,
and we may then conclude that $\varphi=du$.  We then obtain
\begin{align*}
  \lim_{k\to \infty}\int_X |du-du_k|^2\,d\mu &= \int_X
  |du|^2\,d\mu+\lim_{k\to \infty}\left(
  \int_X |du_k|^2\,d\mu -2\int_X \scal{du}{du_k}\,d\mu\right) \\
  & = 2\int_X |du|^2\,d\mu-2 \int_X \scal{du}{\varphi}\,d\mu=0
\end{align*}
by the weak convergence.
\end{remark}

\section{Metric Laplace operator}
\label{SecconstLapl}

In this section we construct a metric Laplace operator $\Delta_X$ on the metric
measure space $(X,d,\mu)$. Recall that a Dirichlet form $\cE$ is
strongly local if whenever $u,v$ are in the domain of $\cE$ and $u$ is
constant on the support of $v$, then $\cE(u,v)=0$. Having a strongly
local Dirichlet form at one's disposal it is rather standard argument
to construct an operator associated to the form. Most of the statements (without
detailed proofs)
can be found in the book of Fukushima, Oshima and Takeda \cite{FOT},
but we provide complete proofs for the reader's convenience.  Since
this operator plays the role of the Laplace operator on $X$, we shall
denote it by $\Delta_X$. Setting
\begin{align*} {\rm Dom}(\Delta_X) & = \left\{ u\in N^{1,2}(X,\mu):
    \textrm{there exists } f\in L^2(X,\mu) \right. \\
  & \qquad \textrm{ such that }
    \cE(u,v)=-\int_X fv\,d\mu \left.\textrm{ for all }v\in N^{1,2}(X,\mu)
  \right\},
\end{align*}
the Laplace operator is defined by
\[
\Delta_X u=f.
\]

We summarize the main properties of this operator in the following theorem.  
The main point is to construct the resolvent operator
$R_\al$, i.e. an operator that gives for any $\al>0$ the formal
solution of the problem
\begin{equation}\label{resolvent}
(\al-\Delta_X)u=f,
\end{equation}
and to deduce from this the main properties of $\Delta_X$.

\begin{theorem}\label{defResolv}
For each $\al>0$, there is an injective bounded linear operator 
$R_\al:L^2(X,\mu)\to  N^{1,2}(X,\mu)$ such that for all $v\in N^{1,2}(X,\mu)$
\[
\int_X fv\, d\mu = \cE_\al(R_\al f,v) = \cE(R_\al f,v)+\al (R_\al f, v)_2.
\]
This operator satisfies:
\begin{enumerate}
\item
for any $f\in L^2(X,\mu)$, $\|R_\al f\|_2\leq \frac{1}{\al}\|f\|_2$;
\item for any $\al,\beta>0$, $R_\al(L^2(X,\mu))=R_\beta(L^2(X,\mu))$,
  and the resolvent equation holds true
\begin{equation}\label{resolventIdentity}
R_\al f -R_\beta f=(\beta-\al)R_\al R_\beta f
\end{equation}
for all $f\in L^2(X,\mu)$;
\item 
for any $f\in L^2(X,\mu)$, we have the following limit in the $L^2(X,\mu)$-norm;
\begin{equation}\label{strongCont}
\lim_{\al\to\infty} \al R_\al f=f
\end{equation}
\end{enumerate}
Properties 2. and 3. imply that $R_\al(L^2(X,\mu))$ is dense in
$L^2(X,\mu)$. In addition, 
\[
{\rm Dom}(\Delta_X)=R_\al(L^2(X,\mu))
\]
for any $\al>0$, and for $u=R_\al f$, $\Delta_X u:=\al u-f$ is
independent of $\alpha$.
\end{theorem}

\begin{proof}
  Let us fix $f\in L^2(X,\mu)$. Then we can define the linear operator
  $T_f:N^{1,2}(X)\to\R$ by $T_f(v)=(f,v)_2:=\int_X f\, v\, d\mu$. 
  We have that
\[
|T_f(v)|\le \|f\|_2 \, \|v\|_2\le 
\frac{\|f\|_2}{\sqrt{\al}}\cE_\al(v,v)^{1/2}. 
\]
Therefore $T_f$ is a bounded linear operator on the Hilbert space
$(N^{1,2}(X,\mu), \cE_\al)$, so by the Riesz representation theorem,
there exists an element of $N^{1,2}(X,\mu)$, denoted by $R_\al f$,
such that $T_f(v)=\cE_\al(R_\al f,v)$. The map $R_\al:L^2(X,\mu)\to
N^{1,2}(X,\mu)$ defined above is linear by the linearity of the defining
operator $f\mapsto T_f$.

Since
\[
\al (R_\al f,v)_2=\cE_\al(R_\al f,v)-\cE(R_\al f,v) =(f,v)_2-\cE(R_\al f,v),
\]
choosing $v=R_\al f$ and applying H\"older's inequality, we see that
\[
0\le \al \|R_\al f\|^2_2= \al (R_\al f,R_\al f)_2=(f,R_\al
f)_2-\cE(R_\al f,R_\al f) \le (f,R_\al f)_2 \le \|f\|_2 \|R_\al
f\|_2. \label{e1.5}
\]
Thus we obtain Claim~1 of the theorem, namely,
\[ 
\al \| R_\al f\|_2\le \|f\|_2. \label{e2}
\]
Thus $R_\al$ as an operator mapping $L^2(X,\mu)$ to $L^2(X,\mu)$ is
bounded with image in $N^{1,2}(X,\mu)\subset L^2(X,\mu)$ and its
operator norm given by
\begin{equation}
\|R_\al\|:=\|R_\al\|_{L^2\to L^2} 
\le \frac{1}{\al}. \label{e3}
\end{equation}

We now prove the resolvent equation \eqref{resolventIdentity}. Let us
take $f\in L^2(X,\mu)$ and $v\in N^{1,2}(X,\mu)$. Then
\begin{align*}
\cE_\al (R_\al f-R_\beta f+(\al-\beta)R_\al R_\beta f,v) & =
       \cE_\al(R_\al f,v)-\cE_\al(R_\beta f,v)+(\al-\beta )\cE_\al(R_\al R_\beta f,v) \\
    =&(f,v)_2-\cE(R_\beta f,v)-\al(R_\beta f,v)_2+(\al-\beta)(R_\beta f,v)_2 \\
    =& (f,v)_2-\cE_\beta (R_\beta f,v)=0.
\end{align*}
This means that for $f\in N^{1,2}(X,\mu)$, and then by density also
for $f\in L^2(X,\mu)$, we have the identity
\[
   R_\al f-R_\beta f +(\al-\beta )R_\al R_\beta f=0.
\]
Moreover, if we consider $f\in N^{1,2}(X,\mu)$, we have (denoting 
$\cE_\al(f,f)^{1/2}=:\Vert f\Vert_\al$)
\begin{align*}
  \al\|\al R_\al f-f\|^2_2 &\le \cE_\al(\al R_\al f-f,\al R_\al f-f) \\
  & =\al^2\cE_\al(R_\al f, R_\al f)+\Vert f\Vert_\al^2-2\al\cE_\al(R_\al f,f)\\
  & = \al^2(f,R_\al f)_2+\cE(f,f)-\al\|f\|^2_2.
\end{align*}
By H\"older's inequality and by using \eqref{e3}, we also get 
\begin{equation}\label{e4}
  \al (f,R_\al f)_2-\|f\|^2_2 \le \al \|R_\al f\|_2\, \|f\|_2-\|f\|^2_2
  \le 0. 
\end{equation}
Therefore, 
\[
\lim_{\al\to\infty}\|\al R_\al f-f\|_2 \leq \lim_{\al\to\infty}\sqrt{\frac{\cE(f,f)}{\al}}  
=0.
\]
To extend this limit to be valid for any $f\in L^2(X,\mu)$, we use the
boundedness of $R_\al$, by fixing $f_\varepsilon\in N^{1,2}(X,\mu)$
such that $\|f-f_\varepsilon\|_2\leq \varepsilon$.  In this way we get
that
\[
\|\al R_\al f-f\|_2\leq  \|\al R_\al f_\varepsilon -f_\varepsilon\|_2
+\al \| R_\al (f-f_\varepsilon)\|_2 +\|f-f_\varepsilon\|_2
\leq  \|\al R_\al f_\varepsilon -f_\varepsilon\|_2+2\varepsilon,
\]
and hence
\[
 \limsup_{\al\to0}\|\al R_\al f-f\|_2\leq 2\varepsilon.
\]
From this Claim~3 of the theorem follows since $\varepsilon$ was arbitrary.

We have now proved that $R_\al$ is a strongly continuous resolvant
(see \cite{FOT}) for any $\al>0$. Let us next prove injectivity of
$R_\al$.  Suppose $f\in L^2(X,\mu)$ is such that $R_\beta f=0$ for
some $\beta >0$. Then by the resolvant equation
(\ref{resolventIdentity}),
\[
  0=R_\al f-R_\beta f+(\al-\beta )R_\al R_\beta f=R_\al f,
\]
that is, $R_\al f=0$ for every $\al>0$. Now by equation~\eqref{strongCont}, we see that
$f=0$, that is, $R_\al$ is injective.
We can therefore define the inverse map $R_\al^{-1}:R_\al(L^2(X,\mu))\to L^2(X,\mu)$.
We claim that 
\[
{\rm Dom}(\Delta_X)=R_\al(L^2(X,\mu)), \qquad \Delta_X u=\al u-R_\al^{-1}u.
\]
For this definition to be consistent, we first show that the set
$R_\al(L^2(X,\mu))$ and the operator $A_\al u:=\al u-R_\al^{-1}u$ do
not depend on $\al$.  By the resolvent equation~\eqref{resolventIdentity}, 
\[
     R_\beta f=R_\al(f+(\al-\beta)R_\beta f).
\]
Therefore, for every $f\in L^2(X,\mu)$, $R_\beta f\in
R_\al(L^2(X,\mu))$, and hence 
\[
   R_\beta (L^2(X,\mu))\subset R_\al(L^2(X,\mu)).
\] 
By the symmetry of the argument, we have the
required result $R_\al(L^2(X,\mu))=R_\beta(L^2(X,\mu))$. Let us write
$D=R_\beta(L^2(X,\mu))$.

If $u\in D$ and $\al,\beta >0$, then $A_\al u-A_\beta
u=(\al-\beta)u -R_\al^{-1}u+R_\beta^{-1}u$. Therefore,
\[
  R_\al(A_\al u-A_\beta u)=\al R_\al u-\beta R_\al u-u+R_\al R_\beta ^{-1}u.
\]
On the other hand, since $D=R_\beta(L^2(X,\mu))$,
there exists $f\in L^2(X,\mu)$ such that $R_\beta f=u$. Hence we have,
by the resolvent equation (\ref{resolventIdentity}), that 
\begin{align*}
  R_\al(A_\al u-A_\beta u) & = \al R_\al R_\beta f-\beta R_\al R_\beta f-R_\beta f+R_\al f \\
          & = R_\al f-R_\beta f+(\al-\beta)R_\al R_\beta f=0.
\end{align*}
By injectivity of $R_\al$, we see that $A_\al u-A_\beta u=0$, i.e.
$A_\al u=A_\beta u$. 

Let us now show that $R_\al(L^2(X,\mu))\subset {\rm
  Dom}(\Delta_X)$. Let $u\in R_\al(L^2(X,\mu))$. Then there exists
$f\in L^2(X,\mu)$ such that $u=R_\al f$. The identity $\cE_\al (R_\al
f,v)=(f,v)_2$ for any $v\in N^{1,2}(X,\mu)$ can be written out as
follows
\begin{align*}
  \int_X \scal{du}{dv}\,d\mu=&\cE(u,v)=\cE_\al(u,v)-\al(u,v)_2\\
  &=\cE_\al(R_\al f,v)-\al(u,v)_2\\
  &=(f,v)_2-\al(u,v)_2=-\int_X (\al u-f)v\,d\mu
\end{align*}
for all $v\in N^{1,2}(X,\mu)$. This simply means that $u\in {\rm
  Dom}(\Delta_X)$ and that $\Delta_X u=\al u-f=\al u-R^{-1}_\al
u=A_\al u$.

For the reverse inclusion, ${\rm Dom}(\Delta_X) \subset
R_\al(L^2(X,\mu))$, let us consider $u\in {\rm Dom}(\Delta_X)$. Thus
there exists $f\in L^2(X,\mu)$ such that for all $v\in N^{1,2}(X,\mu)$
we have
\[
\int_X \scal{du}{dv}\, d\mu=-\int_X fv\,d\mu.
\]
Then consider $w:=R_\al(\al u-f)$; we obtain that
\[
\cE_\al(w,v)=(\al u-f,v)_2=\al \int_X uv\,d\mu -\int_X fv\,d\mu=\al
\int_X uv\,d\mu +\cE(u,v)=\cE_\al(u,v),
\]
that is $w=u$, which means that $u\in R_\al(L^2(X,\mu))$. 
The identity $f=\Delta_X u=\al u-R_\al^{-1}u$ follows easily.

In addition to the density of ${\rm Dom}(\Delta_X)$ in $L^2(X,\mu)$,
we also have that ${\rm Dom}(\Delta_X)$ is dense in
$N^{1,2}(X,\mu)$. In fact, by \eqref{e3} and \eqref{e4}, for any $f\in
N^{1,2}(X,\mu)$, we have that 
\begin{align*}
  \| \al R_\al f-f\|_\al^2 = & \cE_\al(\al R_\al f-f, \al R_\al f-f)=
  \al^2 \cE_\al(R_\al f,R_\al f) -2\al \cE_\al (R_\al f,f)+\cE_\al(f,f) \\
  & = \al^2(f,R_\al f)_2 -\al(f,f)_2 +\cE(f,f)\leq \cE(f,f),
\end{align*}
that is the sequence $(\al R_\al f-f)_\al$ is bounded in
$N^{1,2}(X,\mu)$.  Therefore, for any sequence of positive real
numbers $(\al_n)_n$ so that $\lim_n\al_n=\infty $, the corresponding
sequence of functions $\al_nR_{\al_n}f-f$ is a bounded sequence, and
hence by Mazur's lemma we have a sequence of convex combinations
converging in $N^{1,2}(X,\mu)$;
\[
  \left(\sum_{i=n}^{N(n)}\lambda _{i,n}\al_iR_{\al_i}f\right) -f \to w\in N^{1,2}(X,\mu).
\]
On the other hand, $\lim_{\al\to\infty}\al R_\al f=f$ in
$L^2(X,\mu)$. Thus $w=0$ $\mu$-a.e.~in $X$, and hence by the fact that
$w\in N^{1,2}(X,\mu)$ we know that $w=0$ $p$-q.e. in $X$. Therefore,
it must be that $w=0$. Observe that the sequence of convex
combinations $\sum_{i=n}^{N(n)}\lambda _{i,n}\al_iR_{\al_i}f$ lies in
${\rm Dom}(\Delta_X)$ and converges to $f\in N^{1,2}(X,\mu)$, so the
proof is completed.
\end{proof}

We can now give the definition of a Cheeger harmonic function in the
obvious way.

\begin{definition}
  A function $u\in N^{1,2}(\Omega,\mu)$ is said to be \emph{Cheeger
    harmonic} (referred to in this paper as harmonic) if
\[
\int_\Omega \scal{du}{dv}\, d\mu =0 
\]
for all $v \in N^{1,2}_0(\Omega,\mu)$, i.e. $u$ is harmonic if and only if
$u\in{\rm Dom}(\Delta_\Omega)$ and $\Delta_\Omega u=0$. Here $\Delta_\Om$
is the operator defined in Remark~\ref{Dirich-Dom} below.
\end{definition}
  
\begin{remark}
  The notion of Cheeger harmonicity refers to the fact that we are
  using the Cheeger differentiable structure. This notion has been
  previously considered in the paper \cite{KRS}, where Lipschitz
  regularity of Cheeger harmonic functions has been investigated. We
  also underline that Cheeger harmonicity can be equivalently be given
  in terms of a minimizer of the Dirichlet energy: $u$ is Cheeger
  harmonic if and only if for any ball $B_r$
\[
\int_{B_r} |du|^2\,d\mu\leq
\int_{B_r} |dv|^2\,d\mu,
\]
for all $v$ such that $v-u\in N^{1,2}_0(B_r,\mu)$. 
\end{remark}

\begin{remark}\label{Dirich-Dom}
  Let $\Om\subset X$ be a bounded domain satisfying a
  $(1,2)$-Poincar\'e inequality with ${\rm Cap}_2(X\setminus
  \Om)>0$. The previous construction of $\Delta_X$ can also be used to
  construct a Laplace operator on the subdomain $\Om$. There are
  essentially two different Laplace operators; the first is just the
  restriction of $\Delta_X$ to $\Om$ and is defined by
\begin{align*}
{\rm Dom}(\Delta_\Om) & =\left\{
u\in N^{1,2}(\Om,\mu): \textrm{there exists } f\in L^2(\Om,\mu) \textrm{ such that } \right. \\
& \qquad \left. \int_\Om \scal{du}{dv}\,d\mu=
-\int_\Om fv\,d\mu \textrm{ for all } v\in N^{1,2}_0(\Om,\mu)
\right\},
\end{align*}
and the operator is given by
\[
\Delta_\Om u=f.
\]
The second alternative, adapted to the inhomogeneous Dirichlet problem, 
is the operator defined by
\begin{align*}
{\rm Dom}(\Delta^D_\Om) & = \left\{
u\in N^{1,2}_0(\Om,\mu): \textrm{there exists } f\in L^2(\Om,\mu) \textrm{ such that } \right. \\
& \qquad \left. \int_\Om \scal{du}{dv}\, d\mu=
-\int_\Om fv\, d\mu \textrm{ for all } v\in N^{1,2}_0(\Om,\mu)
\right\},
\end{align*}
and
\[
\Delta^D_\Om u= f.
\]
To define the latter operator, the previous procedure has to be
modified by considering the Hilbert space $N^{1,2}_0(\Om,\mu)$ with
the inner product $\cE_\al$ for all $\al >0$, to obtain the resolvent operator
$R_\al^0:N^{1,2}_0(\Om,\mu)\to N^{1,2}_0(\Om,\mu)$ with
\[
\cE_\al(R_\al^0 f,v)=(f,v)_2
\] 
whenever $v\in N^{1,2}_0(\Om,\mu)$. Since the vector subspace
$\text{Lip}_0(\Om)$ of $N^{1,2}_0(\Om,\mu)$ is also a dense subspace
of $L^2(\Om,\mu)$, we may extend $R_\al^0$ to be an injective map from
$L^2(\Om)$ to $N^{1,2}_0(\Om)$.  These properties, like the one proved
for $R_\al$ in Theorem~\ref{defResolv}, are the essential properties
for the definition of the operator $\Delta^D_\Om$.

It is easy to verify that the operator $\Delta_\Om$ is the restriction
of $\Delta_X$ to $\Om$ in the following sense: If $u\in\text{Dom}(\Delta_X)$,
then 
\[
\Delta_\Om (u_{|\Om})=(\Delta_X u)_{|\Om}.
\]
On the other hand, the operator $\Delta^D_\Om$ is the restriction of $\Delta_\Om$ to
the space $N^{1,2}_0(\Om,\mu)$, that is 
\[
{\rm Dom}(\Delta_\Om^D)={\rm Dom}(\Delta_\Om)\cap N^{1,2}_0(\Om,\mu) 
\]
with $\Delta_\Om^D u=\Delta_\Om u$ for $u\in\text{Dom}(\Delta_\Om)\cap
N^{1,2}_0(\Om,\mu)$.
\end{remark}

\subsection{Measure-valued Laplace operator}
\label{sect:measureLap}

Let $\Omega$ be a domain in $X$. We give the following definition of
the measure-valued Laplace operator $\mathscr{D}_\Om$ on $\Omega$. By
$\mathscr{M}_b(\Omega)$ we denote the space of all bounded signed
Borel measures on $\Omega$, i.e. $\nu\in \mathscr{M}_b(\Omega)$ is a
real-valued signed Borel measure on $\Omega$ with bounded total
variation 
\[
|\nu|(\Om)=\sup \left\{ \int_\Om \varphi\,d\nu: \varphi \in
  \Lip_c(\Om),\, \|\varphi\|_\infty \leq 1 \right\}<\infty.
\]
We remark that to compute the total variation of a measure we test in the space $\Lip_c(\Om)$ of Lipschitz functions on $\Om$ with compact support instead of the space $C_c(\Omega)$ of continuous functions with compact support; we may do this since ${\rm Lip}_c(\Omega)$ is clearly dense in $C_c(\Omega)$.

We define
\begin{align}\label{Laplacian}
{\rm Dom}(\mathscr{D}_\Om)  = &\left\{
 u\in N^{1,2}(\Om,\mu): \textrm{ there exists } \nu\in
  \mathscr{M}_b(X) \textrm{ such that } \right. \nonumber \\
& \qquad\qquad\qquad\qquad \left. \cE(u,v)=-\int_\Om v\,d\nu \textrm{ for all } v\in
  {\rm Lip}_c(\Om) \right\},
\end{align}
and then we set 
\[
\mathscr{D}_\Om u=\nu.	 
\] 

\begin{example}\label{example1}
{\rm 
As an example, we can consider the Euclidean space $(\Rn,\|\cdot\|)$ and modify its
metric structure in two ways, which essentially lead to the same
metric measure structure.  We fix $\Om\subset \Rn$ an open set with
regular boundary and we can modify either the measure by considering
$d\mu=(1+\chi_\Om)d{\cal L}^n$, or the differential structure
$du=(1+\alpha\chi_\Om)\nabla u$, where $\alpha=\sqrt{2}-1$. 

In both cases we have for $u,v\in
C^2_c(\Rn)$
\begin{align*}
\int_\Rn \scal{du}{dv}\,d\mu =& \int_\Rn \nabla u\cdot \nabla v\,dx
+\int_\Om \nabla u\cdot \nabla v\,dx  \\
=&-\int_\Rn v\Delta u dx -\int_\Om v\Delta u\,dx +
\int_{\partial \Om} v\nabla u\cdot \nu_\Om\, d{\cal H}^{n-1}.
\end{align*}
Then $u\in {\rm Dom}(\Delta_{\Rn})$ if and only if $\nabla u\cdot
\nu_\Om=0$ on $\partial \Om$ and $\Delta u\in L^2(\Rn)$. In addition,
in the case $\mu=(1+\chi_\Om) {\cal L}^n$ with the standard
differential structure we also have
\[
\Delta_{\Rn} u= \Delta u.
\]
In the second case, where $\mu={\cal L}^n$ and $du=(1+\alpha\chi_\Om)\nabla
u$, we obtain
\[
\Delta_{\Rn} u= (1+\chi_\Om)\Delta u.
\]
In a similar fashion, ${\rm Dom}(\mathscr{D}_\Rn)$ is given by those
functions $u$ for which $\Delta u\in L^1(\Rn)$ and the trace of
$\nabla u\cdot \nu_\Om\in L^1(\partial \Om,{\cal H}^{n-1})$, and
\[
\mathscr{D}_\Rn u =\Delta u \mu - \nabla u\cdot \nu_\Om {\cal
  H}^{n-1}\res \partial \Om.
\]
}
\end{example}

It can be verified that
$\text{Dom}(\mathscr{D}_\Om)$ is a vector space and that
$\mathscr{D}_\Om$ is linear. We wish to expand the class of test
functions in the definition of the domain
$\text{Dom}(\mathscr{D}_\Om)$ from $\Lip_c(\Om)$ to allow for
test-functions $v$ in $N^{1,2}_0(\Om,\mu)$, see
Proposition~\ref{hopeful}.  For that, we need the following lemma.

\begin{lemma}\label{nulls}
If $E\subset \Om$ is a Borel set such that $\Capa_2(E)=0$, then for every 
$u\in {\rm Dom}(\mathscr{D}_\Om)$, $|\mathscr{D}_\Om u|(E)=0$.
\end{lemma}
 
\begin{proof}
  By the Jordan decomposition theorem, the measure $\mathscr{D}_\Om u$
  can be decomposed into its positive and negative parts,
  $\mathscr{D}^+_\Om u$ and $\mathscr{D}^-_\Om u$; this means that we
  can decompose $\Om$ into two disjoint Borel sets $\Om=\Om^+\cup
  \Om^-$ in such a way that $\mathscr{D}_\Om u(B)\geq0$ for every
  $B\subset \Om^+$ and $\mathscr{D}_\Om u(B)\leq 0$ for every
  $B\subset \Om^-$. Hence we may, without loss of generality, consider
  $E\subset \Om^+$; in fact we can decompose $E=E^+\cup E^-$ and use
  the monotonicity of capacity. Further, we may also assume that $E$
  is a compact set, since as Radon measures both $\mathscr{D}_\Om^+u$
  and $\mathscr{D}_\Om^-u$ are inner measures and $E$ is a Borel set.

  Since $\text{Cap}_2(E)=0$, we have also that the relative capacity
  $\text{Cap}_2(E,\Om)$ is zero. This can be seen by multiplying those
  Lipschitz test-functions which were used for computing
  $\text{Cap}_2(E)$ by another Lipschitz function $\eta$ which is $1$
  on a neighborhood of the compact set $E$ and has compact support in
  $\Om$. We can then find a sequence of Lipschitz functions
  $(\pip_i)_i$ so that $0\le \pip_i\le 1$ on $X$, $\pip_i=1$ on $E$,
  and $\Vert\pip_i\Vert_{N^{1,2}(X)}\le 2^{-i}$, and $\pip_i$ are
  compactly supported in $\Om$. We may assume that the sequence
  $(\pip_i)_i$ converges pointwise to zero outside of the compact set
  $E$ (we can do so by choosing $\pip_i$ to have support in the open
  set $\bigcup_{x\in E}B(x,1/i)$). We have
\begin{align*}
  \left\vert\int_X \pip_i\, d\mathscr{D}_\Omega u\right\vert=
  &\left\vert\int_X \langle du,d\pip_i\rangle\, d\mu
  \right\vert\le   \left(\int_X|du|^2\,d\mu\right)^{1/2} \left(\int_X |d\pip_i|^2\,d\mu\right)^{1/2} \\
  \le & \left(\int_X|du|^2\, d\mu \right)^{1/2}\ \Vert
  \pip_i\Vert_{N^{1,2}(X,\mu)},
\end{align*} 
which tends to 0 as $i\to \infty$.

On the other hand, since $\pip_i$ are all bounded by $1$ and
$|\mathscr{D}_\Om|(X)<\infty$, by the Lebesgue dominated convergence
theorem we have 
\[
\lim_{i\to\infty}\int_X\pip_i\, d\mathscr{D}_\Om
u=\mathscr{D}_\Om u(E)=\mathscr{D}^+_\Om u(E).
\]
A similar argument shows that $\mathscr{D}^-_\Om u(E)=0$, and hence
the proof follows.
\end{proof}

\begin{remark}
  The requirement that $E$ is a Borel set in the above lemma is not a
  serious restriction, because if $E\subset\Om$ is a set with
  $\text{Cap}_2(E)=0$, then there is a Borel set $E_0$ with $E\subset
  E_0\subset\Om$ such that $\text{Cap}_2(E_0)=0$.
\end{remark}


The following proposition tells us that we do not have to restrict
ourselves to having test-functions $v$ only in $\Lip_c(X)$
in~\eqref{Laplacian}. 

\begin{proposition}\label{hopeful}
  Let $u\in{\rm Dom}(\mathscr{D}_\Om)$. Then for every $v\in
  N^{1,2}_0(\Om,\mu)\cap L^\infty(\Om,\mu)$ the following holds:
\[
\cE(u,v) = -\int_\Om v\, d\mathscr{D}_\Om u.
\]
\end{proposition}

\begin{proof}
   We first assume that $v$ has compact support in $\Om$.
  Note that by the $(1,2)$-Poincar\'e inequality we can approximate
  compactly supported functions in $N^{1,2}(\Om,\mu)$ by Lipschitz
  functions. So we can find a sequence of compactly supported
  Lipschitz functions $(\pip_i)_i$ on $\Om$ that converge to $v$ in
  the $N^{1,2}(\Om,\mu)$--norm. By passing to a subsequence if
  necessary, we may also assume that $\pip_i\to v$ pointwise outside a
  set of zero $2$-capacity; we refer to \cite{Sh1}.
  Since $v$ is bounded, we can also assume that the approximating
  compactly supported Lipschitz functions $\pip_i$ are also uniformly
  bounded by $M:=\|v\|_\infty$. Applying $\pip_i$ as in~\eqref{Laplacian}, we see that
\[
   \int_\Om \pip_i\, d\mathscr{D}_\Om u \, =\, -\int_\Om \scal{d\pip_i}{du}\, d\mu 
                           \to -\int_\Om \scal{dv}{du}\, d\mu.
\]
By Lemma~\ref{nulls}, we know that $\pip_i\to v$ almost everywhere
with respect to the total variation measure $|\mathscr{D}_\Om u|$. 
By the Lebesgue dominated convergence theorem
applied to the uniformly bounded functions $\pip_i$ with respect to
the positive and negative parts $\mathscr{D}_\Om u^+$,
$\mathscr{D}_\Om u^-$ of the signed Borel measure $\mathscr{D}_\Om u$,
we may conclude that
\[
\int_\Om \pip_i\, d\mathscr{D}_\Om u\to \int_\Om v\, d
\mathscr{D}_\Om u. 
\]
Hence equation~(\ref{Laplacian}) holds for all
compactly supported functions $v\in N^{1,2}(\Om,\mu)\cap
L^{\infty}(\Omega,\mu)$.

%
%
To pass to any $v\in N^{1,2}_0(\Omega,\mu)\cap L^\infty(\Om,\mu)$, we
note that functions in $N^{1,2}_0(\Om,\mu)$ with compact support in
$\Om$ form a dense subclass of $N^{1,2}_0(\Om,\mu)$ (see \cite{Sh2}).
Hence, if $v$ is in $N^{1,2}_0(\Om,\mu)\cap L^\infty(\Om,\mu)$, we can find
a sequence of compactly supported functions $v_i$ from
$N^{1,2}_0(\Om,\mu)\cap L^\infty(\Om,\mu)$ such that $v_i\to v$ in
$N^{1,2}_0(\Om,\mu)$.  As before, we can also ensure that $v_i\to v$
$2$-capacity almost everywhere in $\Om$. Hence
\[
\int_\Om \scal{du}{dv}\, d\mu=\lim_{i\to\infty}\int_\Om \scal{du}{dv_i}\, d\mu
=-\lim_{i\to\infty}\int_\Om v_i\, d\mathscr{D}_\Om u,
\]
and then if $v$ is bounded in $\Om$ we have 
\[
\lim_{i\to\infty}\int_\Om v_i\, d\mathscr{D}_\Om u=\int_\Om v\, d\mathscr{D}_\Om u, 
\]
giving the desired result for all bounded functions in
$N^{1,2}_0(\Om,\mu)$.
\end{proof}

We shall also need the following lemma, which is based on the Lebesgue
decomposition of the measure $\mathscr{D}_\Om$ given by
\[
d\mathscr{D}_\Om u=f_ud\mu+d\mathscr{D}^s_\Om u,
\]
where $f_u=\frac{d\mathscr{D}_\Om u}{d\mu}$ is the absolutely
continuous part and $\mathscr{D}^s_\Om$ the singular part of
$\mathscr{D}_\Om$.

\begin{lemma}\label{Hip-Hip-Hip-Hooray!}
  Let $u\in {\rm Dom}(\mathscr{D}_\Om)$. If the singular part
  $\mathscr{D}^s_\Om$ of $\mathscr{D}_\Om u$ is zero and if the
  Radon--Nikodym derivative $f_u\in L^2(\Om,\mu)$, then $u\in{\rm
    Dom}(\Delta_\Om)$ with $\Delta_\Om u=f_u$.
\end{lemma}

\begin{proof} 
  From the discussion in Section~\ref{SecconstLapl}, if $\mathscr{D}_\Om^s u=0$
  and the absolutely continuous part is represented by $f_u\in L^2(\Om,\mu)$, then
\begin{equation}\label{intByPartsBdd}
   \cE(u,v)=-\int_\Om f_u v\,d\mu
\end{equation}
for all $v\in N^{1,2}_0(\Om,\mu)\cap L^\infty(\Om,\mu)$. When
$f_u\in L^2(\Om,\mu)$, we can use a truncation argument and the
Lebesgue dominated convergence theorem to show that
\eqref{intByPartsBdd} holds for any $v\in N^{1,2}_0(\Om,\mu)$. So we
conclude that $u\in {\rm Dom}(\Delta_\Om)$ and $\Delta_\Om u=f_u$.

%
%
\end{proof}

\begin{remark}
It can be seen that ${\rm Dom}(\Delta_\Om)\subset {\rm Dom}(\mathscr{D}_\Om)$;
moreover, if $u,v\in {\rm Dom}(\mathscr{D}_\Om)$ and $a\in\R$, the following hold true:
\begin{enumerate}
\item ${\rm spt}(\mathscr{D}_\Om u)\subset{\rm spt}(u)$; also,
if $u$ is constant on an open set $U$, then ${\rm spt}(\mathscr{D}_\Om u)\subset \Om\setminus U$;
\item $u+v, au\in{\rm Dom}(\mathscr{D}_\Om)$ with $\mathscr{D}_\Om(u+v)=\mathscr{D}_\Om u+\mathscr{D}_\Om v$ 
and $\mathscr{D}_\Om(au)=a\, \mathscr{D}_\Om u$;
\item if in addition $u$ and $v$ are bounded,
then $uv\in{\rm Dom}(\mathscr{D}_\Om)$ with 
\[
d \mathscr{D}_\Om(uv)= v\, d\mathscr{D}_\Om u+\, u\, d\mathscr{D}_\Om v+\, 2\, \scal{du}{dv}\, d\mu.
\] 
Note here that since $u,v$ are in $N^{1,2}(\Om)$, it follows that they are well-defined up to sets of 
$\text{Cap}_2$-zero; such null sets are not charged by $\mathscr{D}_\Om u$, $\mathscr{D}_\Om v$,
see Lemma~3.6.
\end{enumerate}
\end{remark}

\subsection{Inhomogeneous Dirichlet problem}
\label{sect:InhomogDP}

In this section we consider the inhomogeneous Dirichlet problem on
bounded open domains $\Om$ such that $\mu(X\setminus\Om)>0$; we assume
that a metric space $X$ satisfies a $(1,2)$-Poincar\'e
inequality. More precisely, given two functions $f\in L^2(X,\mu)$ and
$v\in N^{1,2}(X,\mu)$, we wish to find $u\in {\rm Dom}(\Delta_\Om)$
such that 
\begin{equation}\label{inHom}
\left\{
\begin{array}{l}
  \Delta_\Om u =f \ \textrm{ on } \Omega, \\
  \mbox{ } \\
  u-v \in N^{1,2}_0(\Om,\mu) . 
\end{array}
\right.
\end{equation}
By definition of $\Delta_\Om$, we interpret \eqref{inHom} in the weak
sense, i.e. $u$ is a solution of~\eqref{inHom} if $u-v\in N^{1,2}_0(\Om,\mu)$ and
for all $\pip\in N^{1,2}_0(\Om,\mu)$,
\[
\cE(u,\pip)=-\int_\Om f\pip\,d\mu.
\] 
As in the classical linear theory, a solution $u$ to \eqref{inHom} can
be written as the sum of two functions, $u_0$ and $u_1$, where $u_0$
is harmonic in $\Om$ such that $u_0-v\in N^{1,2}_0(\Om,\mu)$ and
$u_1\in N^{1,2}_0(\Om,\mu)$ is a particular solution to the problem 
$\Delta_\Om^D u_1=f$ with $u_1\in N^{1,2}_0(\Om,\mu)$.

The function $u_0$ is constructed in~\cite{KRS} as the minimum of the energy functional
\[
\min_{u-v\in N^{1,2}_0(\Om,\mu)} \int_\Om |du|^2\,d\mu.
\]
For the second part, we use the
functional $F:N^{1,2}_0(\Om,\mu)\to \R$ given by
\[
F(u)=\frac{1}{2}\int_\Om |du|^2 \,d\mu +\int_\Om uf\, d\mu,
\]
which is the sum of a linear functional and a strictly convex
energy. Hence $F$ itself is strictly convex. Then, if $F$ has a
minimum, it is unique and the minimum is the desired solution
$u_1$. To prove the existence, it is enough to use the Sobolev
inequality, i.e. if $u\in N^{1,2}_0(\Om,\mu)$ there exists a constant
$c_s>0$ such that
\[
\| u\|_2 \leq c_s\| du\|_2.
\]
Given that $\Om$ is bounded and $\mu(X\setminus\Om)>0$, the above
Sobolev inequality holds; we refer to \cite{HaKo} and \cite{KiSh} for
the details.  Then, for any $u\in N^{1,2}_0(\Om,\mu)$ we have that,
using the inequality $ab\leq \eps a^2/2 +b^2/2\eps$ with
$a,b,\eps>0$,
\begin{align*}
  F(u)=& \frac{1}{2}\| du\|_2^2 +\int_\Om uf\, d\mu \geq \frac{1}{2}\|
  du\|_2^2 -\|f\|_2 \|u\|_2 \geq
  \frac{1}{2}\| du\|_2^2 -c_s\|f\|_2 \|du\|_2 \\
  \geq& \left( \frac{1}{2} -\frac{\eps c_s}{2}\right) \| du\|_2^2
  -\frac{c_s}{2\eps}\|f\|^2_2.
\end{align*}
If we fix $\eps <1/c_s$, the preceding inequality gives us that $F$ is
bounded from below by $-2^{-1}\eps^{-1}c_s\|f\|_2^2$. Therefore,
\[
m=\inf_{u\in N^{1,2}_0(\Om,\mu)} F(u)
\]
is finite, and in particular, the infimum is a minimum as seen by
taking a minimizing sequence and applying Mazur's lemma. The
minimizing function $u_1$ is a weak solution to the desired equation,
that is
\begin{equation}\label{weakSol}
\int_\Om \scal{du_1}{d\pip}\,d\mu =-\int_\Om f\pip\, d\mu
\end{equation}
for all $\pip \in N^{1,2}_0(\Om,\mu)$. From \eqref{weakSol} it is
immediate to see that $u_1$ is the desired solution; in addition, if
in \eqref{weakSol} we take $\pip=u_1$, we have the Caccioppoli type estimate
\[
\| du_1\|_2\leq c_s \| f\|_2.
\]

\section{Functions of bounded variation and the perimeter measure}
\label{sect:BVPerimeter}

The aim of this section is to study some properties of the perimeter measure of a ball in
metric space. The properties we have in mind are needed in the
characterization of a singular function which will be constructed in
Section~\ref{sect:SingFunct}.

Following \cite{Mr}, the definition of the total variation of a function
$u\in L^1(X,\mu)$ is given by 
\begin{equation}\label{totVar}
  |Du|(X)=\inf \left\{
    \liminf_j \int_X g_{u_j}\, d\mu: u_j\in \Lip_{\textup{loc}}(X,\mu),\ u_j\to u \mbox{ in } L^1_{\textup{loc}}(X,\mu)
  \right\}.
\end{equation}
A function $u$ is said to have bounded variation, that is,
$u\in BV(X,\mu)$, if $|Du|(X)<\infty$. Moreover, a Borel set
$E\subset X$ with finite measure is said to have finite perimeter if
$\chi_E\in BV(X,\mu)$. We denote the perimeter measure of $E$ by
$P(E,X)=|D\chi_E|(X)$. 

To each function of bounded variation we associate a Borel regular
measure, its total variation measure. This measure is defined on every
open set $A\subset X$ using (\ref{totVar}), that is,
\[
|Du|(A)=\inf \left\{ \liminf_j \int_A g_{u_j}\,d\mu: u_j\in
  \Lip_{\textup{loc}}(A,\mu),\ u_j\to u \mbox{ in } L^1_{\textup{loc}}(A,\mu)
\right\}.
\]
We extend this measure to act on any Borel set $B\subset X$ by the
Carath\'eodory construction
\[
|Du|(B)=\inf \left\{
|Du|(A): A \mbox{ open and } B\subset A
\right\};
\]
for more details on this construction in the metric measure setting see~\cite[Theorem 3.4]{Mr}.

An equivalent definition can be also given by way of the Cheeger
differentiable structure as follows: 
\[
|D_cu|(X)=\inf \left\{ \liminf_j \int_X |du_j|\,d\mu: u_j\in
  \Lip_{\textup{loc}}(X,\mu),\ u_j\to u \mbox{ in } L^1_{\textup{loc}}(X,\mu)
\right\},
\]
and we shall say that $u$ has bounded total Cheeger variation if
$|D_cu|(X)<\infty$; a set with Cheeger finite perimeter is a Borel set
$E$ with finite measure such that $|D_c\chi_E|(X)<\infty$. 

By the results contained in \cite{C}, it follows that these two
definitions are equivalent, in the sense that $u$ has bounded
total variation if and only if it has bounded total Cheeger variation.
There exists a constant $c>1$ such that
\[
\frac1{c}|Du|(X)\leq |D_cu|(X)\leq c |Du|(X).
\]
Also using the Cheeger differentiable stucture, we have that $|D_cu|$
defines a finite Radon measure; the argument is similar to the case of
$|Du|$ and so we refer to \cite{Mr} for the proof.

A sequence of Lipschitz functions $(u_j)_j$ is said to converge in
variation to a function $u\in BV(X,\mu)$ if $u_j$ conveges to $u$ in
$L^1_{\textup{loc}}(X,\mu)$ and
\[
\int_X g_{u_j}\,d\mu \to |Du|(X).
\]
The preceding definition of total variation does not specify the
optimal sequence, i.e. the sequence which converges to $u$ in
variation.

It is proved in \cite[Theorem 3.8]{Mr} that the discrete convolution
gives an approximation that is only comparable in variation with the
optimal one. Note that for the optimal sequence,
the vector valued measures $d\vec \mu_j=du_j\, d\mu$, have uniformly
bounded total variation.  So, up to subsequences, they converge to
some vector valued finite measure $\vec \mu_\infty$.

\begin{remark}\label{rmkOpt}
  A sequence $(u_j)_j$ converging to $u$ in variation is optimal not
  only for the variation in $X$, but also for the variation in all 
  open subsets $A$ with $|Du|(\partial A)=0$. In fact, by definition,
  we have that
\[
|Du|(A)\leq \liminf_{j\to\infty} \int_A g_{u_j}\, d\mu, 
\]
but also that
\begin{align*}
  |Du|(X\setminus \overline A)&\leq 
  \limsup_{j\to \infty} \int_{X\setminus \overline A} g_{u_j}\,d\mu \leq 
  \limsup_{j\to \infty} \left( \int_X g_{u_j}\,d\mu -\int_A g_{u_j}\,d\mu
  \right) \\
  &= |Du|(X)-\liminf_{j\to \infty} \int_A g_{u_j}\,d\mu\leq |Du|(X\setminus A).
\end{align*}
The preceding inequalities are indeed equalities if $|Du|(\partial
A)=0$ and so $|Du|(X\setminus A)=|Du|(X\setminus\overline{A})$. Hence
the following two limits exist
\[
   \lim_{j\to \infty} \int_{X\setminus A} g_{u_j}\,d\mu=|Du|(X\setminus A),
\]
and
\begin{equation}\label{limiting}
    \lim_{j\to \infty} \int_A g_{u_j}\, d\mu =|Du|(A).
\end{equation}
\end{remark}

An important tool in the theory of functions of bounded variation is
the coarea formula. The version we work with in the present paper is a
direct consequence of \cite[Proposition~4.2]{Mr}. For any $u\in BV(X,\mu)$ 
and any Borel measurable function $f:X\to \R$, the following identities hold
\begin{equation}\label{coarea}
\int_X f\,d|Du| =\int_\R\int_X f(x)\, d|D\chi_{E_t}|(x)dt
\end{equation}
and 
\[
\int_X f\, d|D_cu| =\int_\R\int_X f(x)\, d|D_c\chi_{E_t}|(x)dt,
\] 
where $E_t=\{u>t\}$, $t\in\R$, is the super-level set of $u$. We point
out that in these formulae, due to the fact that the measures $|Du|$ 
and $|D_cu|$ are not absolutely continuous with respect to the measure
$\mu$, it is important to consider the function $f$ and not an
equivalent representative. Since the perimeter measure does
not charge sets with zero $1$--capacity, we can modify the function
$f$ on such negligible sets. If $u$ is Lipschitz, (\ref{coarea})
can be written as follows
\[
\int_X f g_u\, d\mu =\int_\R\int_X f\,d|D\chi_{E_t}|(x)dt.
\]
This follows by an argument contained in \cite[Theorem~6.2.2]{Ca} and summarized in
the following proposition. We will provide a proof here for the reader's convenience.

\begin{proposition} \label{prop:TotalVarMeasure}
  Let $u\in \Lip(X)$. Then the total variation measure $d|Du|$ is
  given by $g_ud\mu$.
\end{proposition}

\begin{proof}
  Lipschitz continuity of $u$ implies that $|Du|$ is absolutely
continuous with respect to $\mu$ with density given by some function
$G_u$. To see this, note that $u_j=u$ is a possible competitor in the
definition of $|Du|$, and so for every open set $A\subset X$ (and hence
for every set $A\subset X$) we have
\[
   |Du|(A)\le \int_A g_u\, d\mu.
\]
Therefore the density function $G_u\le g_u$ $\mu$-a.e.
%
 
To prove the equality it suffices to
prove that the function $G_u$ is an upper gradient of $u$. We take a
sequence $(u_j)_j$ of Lipschitz functions converging to $u$ in
$L^1(X,\mu)$ and with
\[
\lim_{j\to\infty} \int_X g_{u_j}\,d\mu =|Du|(X).
\] 
The sequence of measures $g_{u_j}d\mu$ is bounded, and so, up to a
subsequence, it converges weakly to a measure $\mu_\infty$ that is
still absolutely continuous with respect to $\mu$,
i.e. $d\mu_\infty=g_\infty \, d\mu$.  
To see that $\mu_\infty$ is absolutely continuous with respect to
$\mu$, it suffices to show that whenever $E\subset X$ is
compact with $\mu(E)=0$, we have $\mu_\infty(E)=0$. To this end, we
note that because $E$ is compact, for every $\eps>0$ we can 
cover $E$ with a finite number of balls $B_i^\eps$ with $\mu(\partial B_i^\eps)=0$
so that the open set $A_\eps=\cup_i B_i^\eps$ contains $E$ and is such
that $\mu(A_\eps)<\eps$ and $\mu(\partial A_\eps)=0$. 
It follows that $|Du|(\partial A_\eps)=0$
because of the absolute continuity of $|Du|$ established
above. Therefore, by Remark~\ref{rmkOpt} we have 
\[
\lim_{j\to \infty} \int_{A_\eps} g_{u_j}\,d\mu = |Du|(A_\eps)\leq \int_{A_\eps} g_u\, d\mu
\leq L\eps,
\]
which implies that $\mu_\infty(A_\eps)\le L\eps$. It follows that
$\mu_\infty(E)=0$. For more general sets $E\subset X$ with $\mu(E)=0$,
there is a Borel set $E_0$, containing $E$, such that
$\mu(E_0)=0$. Because $\mu_\infty$ is a Borel measure,
$\mu_\infty(E_0)$ is the supremum of all $\mu_\infty(K)$, the supremum
taken over all compact sets $K\subset E_0$. Given that $\mu(K)=0$ and
so $\mu_\infty(K)=0$, it follows that $\mu_\infty(E_0)=0$ and so
$\mu_\infty(E)=0$. Then $g_{u_j}$ converges to $g_\infty$ weakly in
$L^1(X,\mu)$.

To summarize, we have $u_j\to u$ in $L^1(X)$ and $g_{u_j}\to g_{\infty}$
weakly in $L^1(X)$. Now, an invocation of the Mazur lemma, together with
\cite[Lemma~3.1]{KaS} shows that $g_\infty$ is a weak upper gradient of $u$;
it follows that $g_u\le g_\infty$ a.e.~in $X$.   

Given that balls have finite $\mu$-measure, when $x\in X$, for almost every $r>0$
we have $\mu(\partial B_r(x))=0$. For such $r>0$,
by~\eqref{limiting} with $A=B_r(x)$,  
\[
|Du|(B_r(x))\le \int_{B_r(x)}g_\infty\, d\mu\le
\liminf_{j\to\infty}\int_{B_r(x)}g_{u_j}\,
d\mu=|Du|(B_r(x))=\int_{B_r(x)} G_u\, d\mu,
\]
and so we get $g_\infty(x)=G_u(x)$ for $x\in X$ that are Lebesuge points for
both $G_u$ and $g_\infty$. It follows that $G_u\le g_u\le g_\infty=G_u$ a.e.~in $X$,
from which the claim follows.

\end{proof}

From now on, we assume that $X$ is also a geodesic space. This
  is not an overly restrictive assumption, since $X$, by the virtue of
  supporting a Poincar\'e inequality and being complete, is a
  quasiconvex space. It follows that in our setting, a bi-Lipschitz
  change in the metric does result in a geodesic space. 

Let us fix a point $x_0\in X$. Since $X$ is assumed to be a geodesic
space, by the results in~\cite{C}, the function $u_{x_0}(x)=d(x,x_0)$
is Lipschitz with $\text{Lip}(u_{x_0})=g_{u_{x_0}}\equiv1$. Moreover,
we may write $B_t(x_0)=\{u_{x_0}<t\}$ and so by the coarea formula
(\ref{coarea}) we obtain for any positive $r>0$ that
\[
\int_0^r P(B_t(x_0),X)\,dt =\mu(B_r(x_0)) <\infty.
\]
Thus the map $t\mapsto P(B_t(x_0),X)$ is a measurable locally
integrable function. This implies that for almost every $r>0$,
\begin{align*}
  P(B_r(x_0),X) & =\lim_{\eps \to 0} \frac{1}{\eps} \int_{r-\eps}^r
  P(B_t(x_0),X)\, dt = \lim_{\eps \to 0}\frac{1}{\eps} \int_r^{r+\eps}
  P(B_t(x_0),X)\, dt \\
  & = \lim_{\eps \to 0} \frac{1}{2\eps}
  \int_{r-\eps}^{r+\eps} P(B_t(x_0),X)\, dt.
\end{align*}
In particular, for almost every $r>0$ the perimeter measure coincides
with the Minkowski content
\begin{equation}\label{MinContInt}
P(B_r(x_0),X)=\lim_{\eps\to 0}
\frac{\mu(B_r(x_0))-\mu(B_{r-\eps}(x_0))}{\eps}. 
\end{equation}
For a ball $B_r(x_0)$ satisfying \eqref{MinContInt}, we can consider the sequence of functions 
$(u_\eps)_{\eps>0}$, where
\begin{equation}\label{optimalseq}
u_\eps(x)=\max\left\{ \min
\left\{\frac{r-d(x_0,x)}{\eps},1\right\},0\right\}= 
\min\left\{\frac{1}{\eps} d(x,X\setminus B_r(x_0)),1\right\}.
\end{equation}
For a such function $u_\eps$, we have that
$g_{u_\eps}=\frac{1}{\eps}\chi_{B_r(x_0)\setminus B_{r-\eps}(x_0)}$
and
\[
\int_X g_{u_\eps}\,d\mu =\frac{1}{\eps} \int_{B_r(x_0)\setminus B_{r-\eps}(x_0)}\, d\mu
\rightarrow P(B_r(x_0),X),
\]
that is, the sequence $u_\eps$ converges to $\chi_{B_r(x_0)}$ in
variation. This also means that the sequence of vector valued measures
$(|du_\eps| d\mu)_\eps$ is equibounded
\[
|D_c u_\eps|(X)=\int_X |du_\eps|\,d\mu \leq c <\infty
\]
for some positive constant $c$. Therefore there exists a subsequence
$\eps_j\to 0$ such that, setting $u_j=u_{\eps_j}$, the sequence of vector-valued
measures $du_j\, d\mu$ is weakly convergent to some vector valued
measure $\vec\mu_\infty$. This measure is absolutely continuous with
respect to both $|D_c\chi_{B_r(x_0)}|$ and
$|D\chi_{B_r(x_0)}|$. Indeed, if, for instance,
$|D_c\chi_{B_r(x_0)}|(E)=0$, where $E\subset X$ is compact, we can
find for every $\eps>0$ an open set $A_\eps\supset E$ such that
$|D_c\chi_{B_r(x_0)}|(\partial A_\eps)=0$ and
$|D_c\chi_{B_r(x_0)}|(A_\eps)<\eps$. Reasoning as in the proof of
Proposition~\ref{prop:TotalVarMeasure} we may conclude that
\[
|\vec \mu_\infty|(E)\leq |\vec \mu_\infty|(A_\eps)\leq
\liminf_{j\to\infty}\int_{A_\eps} |du_j|\,d\mu
=|D_c\chi_{B_r(x_0)}|(A_\eps)<\eps.
\]
We may hence write
\[
\vec\mu_\infty =\nu_{x_0,r} |D_c\chi_{B_r(x_0)}| =\sigma_{x_0,r}
|D\chi_{B_r(x_0)}|
\]
for some vector-valued $|D_c\chi_{B_r(x_0)}|$-measurable function $\nu_{x_0,r}$ and
$|D\chi_{B_r(x_0)}|$-measurable function $\sigma_{x_0,r}$. In particular, the function $\sigma_{x_0,r}$ plays
the role of the normal vector at the boundary of $B_r(x_0)$. In this
context, it is not clear from the definition if it is a unit
vector. For the sake of simplicity, if no confusion may arise, we
simply denote the functions $\nu_{x_0,r}$ and $\sigma_{x_0,r}$ by
$\nu$ and $\sigma$, respectively.

We can summarize the previous construction in the following definition.

\begin{definition}\label{def:regular}
  We shall call a ball $B_r(x_0)$ {\it regular} if the equation
  \eqref{MinContInt} is valid and if there exists a
  sequence $\eps_j\to 0$ such that for the sequence of functions
  $u_j=u_{\eps_j}$, referred to as an {\it optimal sequence} and defined
  in \eqref{optimalseq}, the following hold true:
\begin{enumerate}
\item[(1)] $g_{u_j}\, d\mu$ converges weakly to $d|D\chi_{B_r(x_0)}|$;
\item[(2)] $du_j  \, d\mu$ converges weakly to $\sigma\, d|D\chi_{B_r(x_0)}|$ for
  some $|D\chi_{B_r(x_0)}|$-measurable vector-valued function $\sigma=\sigma_{x_0,r}$.
\end{enumerate} 
\end{definition}

Almost every ball is regular in the sense that
for every $x_0\in X$ and for almost every $r>0$ the ball $B_r(x_0)$ is
regular. Howeover, the vector $\sigma$ is not a priori unique and it
is not clear whether it depends on the sequence $\eps_j$ we 
consider. 

The given notion of regularity relates to interior regularity of a ball. One can also 
consider the notions of outer and two-sided regularity
and obtain that for almost every radius $r>0$ the ball $B_r(x_0)$ has
inner, outer, and two-sided regularity.

\section{Divergence measures and generalized Gauss--Green formulas}
\label{sect:GG}

Here we consider divergence-measure fields, i.e. a class of vector fields
$\vec{F}:X\to\R^k$ belonging to the space $L^2(\Om,\R^k,\mu)$ and for
which $\diver \vec F$ is a measure. In the metric space framework of
the present paper, we generalize some results obtained by Ziemer in
\cite{Z}. 

Previously, Thompson and Thompson in \cite{TT} constructed a
divergence form in the setting of Minkowski spaces, and they proved a
Minkowski space analogue of the Gauss--Green theorem.

The aim of this section is to study the operator $\diver$ on
$\Lip_c(\Om)$ with values in the space of measures, that is, we want
to define for $\vec F\in L^2(\Om,\R^k,\mu)$ the distribution
\begin{equation}\label{div}
    \scal{u}{\diver \vec{F}}:=-\int_\Om \scal{\vec{F}}{du}\, d\mu, 
\end{equation} 
for $u\in \Lip_c(\Om)$. In the following we adopt the notation from \cite{CTZ} and \cite{Z}.

\begin{definition} \label{def:divF}
  We say that $\vec{F}\in L^2(\Om,\R^k,\mu)$ is in the class
  $\DMd(\Om)$ if there is a signed finite Radon measure, denoted by
  $\diver\vec{F}\in \mathscr{M}_b(\Om)$, on $\Om$ such that
\begin{equation}\label{defDivergenceMeas}
\int_\Om u\, d\diver\vec{F}\, =\, -\int_\Om\scal{\vec{F}}{du}\, d\mu
\end{equation}
for all $u\in \Lip_c(\Om)$.
\end{definition}

\begin{remark} \label{rmk:LiptoN}
  By an argument similar to the proof of Lemma~\ref{nulls}, we can
  prove that if $\vec F\in \DMd(\Om)$, then the measure $\diver \vec
  F$ does not charge sets with zero $2$--capacity. Therefore, condition
  \eqref{defDivergenceMeas} can be extended to any $u\in
  N^{1,2}_0(\Om,\mu)\cap L^\infty(\Om,\mu)$.
\end{remark}

Note that when $\vec{F}\in L^2(\Om,\R^k,\mu)$ the operator $T_{\vec
  F}:N^{1,2}_0(\Om,\mu)\to \R$ given by
\[
  T_{\vec F}(u)= \int_\Om \scal{\vec{F}}{du}\, d\mu 
\]
is a bounded linear operator on the Hilbert space
$(N^{1,2}_0(\Om,\mu),\cE_1)$. Therefore, by the Riesz representation
theorem, there exists a function $v\in N^{1,2}_0(\Om,\mu)$ such that
whenever $u\in N^{1,2}_0(\Om,\mu)$, $T_{\vec
  F}(u)=\cE_1(v,u)$. Hence, if $\vec{F}\in \DMd(\Om)$, then
for all $u\in N^{1,2}_0(\Om,\mu)$ we obtain
\[
\int_\Om\scal{\vec{F}}{du}\, d\mu=\int_\Om uv \, d\mu + \int_\Om
\scal{du}{dv}\, d\mu,
\]
that is,
\[
\int_\Om u\, d\diver \vec{F}+\int_\Om u\, v\, d\mu \, =\, -\int_\Om \scal{du}{dv}\, d\mu.
\]
It follows that $v\in {\rm Dom}(\mathscr{D}_\Om)$ with
\[
 d\mathscr{D}_\Om v=-v \, d\mu - d\diver\vec F.
\]
This proves the following lemma.

\begin{lemma}
Given a domain $\Om\subset X$, a map $\vec{F}\in L^2(\Om,\R^k,\mu)$ is in the class 
$\DMd(\Om)$ if and only if
 there exists $v\in {\rm Dom}(\mathscr{D}_\Om)$ such that
 \[
  d\mathscr{D}_\Om v=-v\, d\mu- d\diver\vec{F}
 \]
in the sense of distributions on $N^{1,2}_0(\Om,\mu)$.
\end{lemma}

We can also state the following simple properties of the divergence measure.

\begin{lemma}
  Let $\vec F\in \DMd(\Om)$. Then ${\rm spt}(\diver\vec F)\subset {\rm
    spt}(\vec F)$. Moreover, if $v\in{\rm Dom}(\mathscr{D}_\Om)$, then
  $dv\in\DMd(\Om)$ with $\diver\, dv=\mathscr{D}_\Om v$.
\end{lemma}

\begin{proof}
  The first statement follows by considering $A=\Om\setminus {\rm
    spt}(\vec F)$, so we have that 
  \begin{align*}
  |\diver \vec F|(A)&=\sup \left\{\int_A \varphi\, d\diver \vec F:
    \varphi\in \Lip_c(A),\, \|\varphi\|_\infty\leq 1
  \right\}\\
  &= \sup \left\{ \int_A \scal{\vec F}{d\varphi}\,d\mu: \varphi\in
    \Lip_c(A),\, \|\varphi\|_\infty\leq 1 \right\}=0.
\end{align*}
For the second part, let $v\in {\rm Dom}(\mathscr{D}_\Om)$. Then there
exists a signed finite Radon measure $\mathscr{D}_\Om v\in \mathscr{M}_b(X)$
such that
\[
\int_\Om \scal{dv}{du}\,d\mu=-\int_\Om u\,d\mathscr{D}_\Om v
\]
for all $u\in N^{1,2}_0(\Om,\mu)\cap L^\infty(\Om,\mu)$. From this and by
Remark~\ref{rmk:LiptoN}, we may conclude that $dv\in \DMd(\Om)$ with
$\diver dv=\mathscr{D}_\Om v$, and the claim follows.
\end{proof}

We now state the following two propositions on the Gauss--Green type
integration by parts formula for vector fields in $\DMi(\Om)$, that is
for vector fields $\vec F$ in $L^\infty(\Om,\R^k,\mu)\cap \DMd(\Om)$. 

\begin{proposition}\label{dmiCont}
  Let $\vec{F}\in {\cal DM}^\infty(\Om)\cap C(\Om,\R^k)$ and
  $B_r(x_0)\subset \Om$ be a regular ball. The following Gauss--Green
  formula
\[
\int_{B_r(x_0)} f\,  d\diver\vec{F}+\int_{B_r(x_0)} \scal{\vec F}{df}\, d\mu= 
-\int_\Om f \scal{\vec F}{\sigma_{x_0,r}}\, d|D\chi_{B_r(x_0)}|, 
\]
holds for all $f\in \Lip_c(\Om)$. If the support of $\vec F$ is disjoint from
$\partial B_r(x_0)$, then the requirement that $\vec{F}$ is continuous
can be removed.
\end{proposition}

\begin{proof}
  We can consider an optimal sequence of locally Lipschitz functions
  $(u_j)_j$ converging to $\chi_{B_r(x_0)}$ in variation as in
  \eqref{optimalseq}. Then we have by the Leibniz rule that
\begin{align*}
\int_\Om u_j f\, d\diver\vec F &=-\int_\Om \scal{\vec F}{d(u_jf)}\, d\mu
 =-\int_\Om u_j\scal{\vec F}{df}\, d\mu
-\int_\Om f\scal{\vec F}{du_j}\, d\mu.
\end{align*}
We notice that 
\begin{equation}\label{tendToZero}
  \left|\int_\Om u_j f\, d\diver\vec F -\int_{B_r(x_0)} f\, d\diver\vec F\right|
  \leq \|f\|_\infty \left|\diver \vec F\right|\left(B_r(x_0)\setminus B_{r-\eps_j}(x_0)\right),
\end{equation}
and that the right-hand side of \eqref{tendToZero} tends to 0 as $j\to
\infty$. Thus we may conclude that
\[
\lim_{j\to \infty} \int_\Om u_j f\, d\diver\vec F =\int_{B_r(x_0)} f\, d\diver\vec F.
\]
Also, by the fact that both $\vec{F}$ and $df$ are in $L^\infty(\Om)$,
by an application of the Lebesgue dominated convergence theorem we
obtain
\[
   \lim_{j\to\infty}  \int_\Om u_j \scal{\vec F}{df}\, d\mu =\int_{B_r(x_0)}\scal{\vec F}{df}\, d\mu.
\]
We also have, due to the continuity of $\vec F$, that
\[
   \lim_{j\to \infty} \int_\Om f\scal{\vec F}{du_j} \, d\mu 
      = \int_\Om f\scal{\vec F}{\sigma_{x_0,r}} d|D\chi_{B_r(x_0)}|,
\]
and so the proof is completed.
\end{proof}

\begin{remark}
  We point out that property \eqref{tendToZero} is a consequence of
  the choice of an optimal sequence $(u_j)_j$ to be an inner
  approximation of the characteristic function $\chi_{B_r(x_0)}$. If
  we chose, for instance, an outer approximation, then the preceding
  integration by parts formula would be as follows
\[
\int_{\overline B_r(x_0)} f\, d\diver\vec{F}+\int_{B_r(x_0)}
\scal{\vec F}{df}\, d\mu= -\int_\Om f\scal{\vec F}{\tilde
  \sigma_{x_0,r}}\, d|D\chi_{B_r(x_0)}|,
\]
for all $f\in \Lip_c(\Om)$, where $\tilde\sigma_{x_0,r}$ is the
density of the vector-valued measure obtained as a weak limit by way
of the gradients of this new sequence as in
Definition~\ref{def:regular}.

We also point out that the previous proposition can be extend to more
general sets $E\subset\Om$ with finite perimeter whenever a Minkowski
content characterization of the perimeter, analogous to
\eqref{MinContInt}, holds. In this case, the boundary of $E$ has to be
considered as the essential, or the measure-theoretic, boundary of
$E$, i.e. the set of all points at which the density of $E$ is neither
0 nor 1.
\end{remark}

We prove the following main theorem of this section, which is a generalization of
Proposition~\ref{dmiCont}, without requiring continuity of the vector
field. This theorem should be thought of as the generalization
of the Gauss--Green theorem of the Euclidean setting.

\begin{theorem}\label{Question1}
  Let $\vec F\in {\cal DM}^\infty(\Om)$ and let $B_r(x_0)\subset \Om$
  be a regular ball. Then the following extended Gauss--Green formula
\begin{equation}\label{IntByParts}
  \int_{B_r(x_0)} f\, d\diver \vec F +\int_{B_r(x_0)} \scal{\vec F}{df}\, \dmu=
  \int_\Om f\, (\vec F\cdot \nu)^-_{\partial B_r(x_0)}\,d|D\chi_{B_r(x_0)}|, 
\end{equation}
holds for all $f\in N^{1,2}(\Om,\mu)\cap L^{\infty}(\Om,\mu)$, where
$(\vec F\cdot \nu)^-_{\partial B_r(x_0)}$ is the interior normal trace
of $\vec F$ on $\partial B_r(x_0)$.
\end{theorem}

\begin{proof}
  We use the optimal sequence $(u_j)_j$ defined in~\eqref{optimalseq}.
  Then, as in the proof of Proposition~\ref{dmiCont}, 
  by the definition of $\text{div}\vec{F}$
  (Definition~\ref{def:divF}) and the Lebesgue dominated convergence theorem
\begin{align*}
  \lim_{j\to\infty} \int_\Om f\scal{\vec F}{du_j}\,d\mu
  &= \lim_{j\to\infty}\left(\int_\Om\scal{\vec F}{d(u_jf)}\, d\mu-\int_\Om u_j\scal{\vec F}{df}\, d\mu \right) \\
  &= -\lim_{j\to\infty}\left(\int_\Om u_j f\, d\diver{\vec F}+\int_\Om u_j\scal{\vec F}{df}\, d\mu\right) \\
&= -\int_{B_r(x_0)}f\, d\diver{\vec F} -\int_{B_r(x_0)}\scal{\vec F}{df}\, d\mu.
\end{align*}
For the sequence $(L_j)_j$ of operators given by
\[
L_j(f):=\int_\Om f\scal{\vec F}{du_j}\,d\mu,
\]
we have that $|L_j(f)|\leq C \|\vec F\|_\infty \|f\|_\infty$, where
the positive constant $C$ is given by
\[
C= \sup_{j\in \N} \int_\Om |du_j|\,d\mu <\infty.
\]
Indeed, $C$ is finite since the ball $B_r(x_0)$ has finite
perimeter. In particular, $C$ is independent of both $f$ and
$\vec F$ and so, by the above argument, the operator
\[
L(f):=\lim_{j\to \infty} L_j(f)
\]
is bounded over $\Lip_c(\Om)$ and admits an extension to
$C_c(\Om)$. This in turn implies that there exists a measure
$\nu\in\mathscr{M}_b(\Om)$ such that for any $f\in C_c(\Om)$
\[
L(f)=\int_\Om f\,d\nu.
\]
The measure $\nu$ is concentrated on $\partial B_r(x_0)$; in fact take
any compact set $K$ such that $K\cap \partial B_r(x_0)=\emptyset$, an
open set $A\supset K$ such that ${\rm dist}(A,\partial B_r(x_0))>0$,
and take $\eps_j< {\rm dist}(A,\partial B_r(x_0))$. Then, since ${\rm
  spt}(du_j)\cap A=\emptyset$, we obtain for any $f\in \Lip_c(A)$,
\[
\int_\Om f\,d\nu =\lim_{j\to\infty} \int_\Om f\scal{\vec F}{du_j}\,d\mu =0,
\]
that is $|\nu|(K)=|\nu|(A)=0$. This property extends to any Borel set
$E$ such that $E\cap \partial B_r(x_0)=\emptyset$ since
\[
|\nu|(E)=\sup_{K\subset E}|\nu|(K)=0.
\]
Also $\nu$ can be seen to be absolutely continuous with respect to
$|D_c\chi_{B_r(x_0)}|$; indeed, if $E$ is a Borel set such that
$|D_c\chi_{B_r(x_0)}|(E)=0$, then there exists an open set $A_\eps$
such that $|D_c\chi_{B_r(x_0)}|(A_\eps)<\eps$. Fix a compact set
$K\subset E$ and an open set $A\supset K$ such that $\bar A\subset
A_\eps$. Then, for any $f\in \Lip_c(A)$ with $\|f\|_\infty\leq 1$ we
have that
\[
\left| \int_A f\scal{\vec F}{du_j}\,d\mu \right|
\leq \limsup_{j\to \infty} \| \vec F\|_\infty \int_A |du_j|\,d\mu
\leq \| \vec F\|_\infty |D_c\chi_{B_r(x_0)}|(\bar A)
<\eps \| \vec F\|_\infty,
\]
that is $|\nu|(A)<\eps$. Therefore, since $\eps$ is arbitrary,
$|\nu|(K)=0$. Finally, by taking the supremum over $K\subset E$, we
obtain that $|\nu|(E)=0$, and hence $\nu$ is absolutely continuous
with respect to $|D_c\chi_{B_r(x_0)}|$.

To conclude, there exists $(\vec F\cdot\nu)^-_{\partial B_r(x_0)}\in
L^1(|D\chi_{B_r(x_0)}|)$ such that
\[
L(f)=-\int_\Om f (\vec F\cdot \nu)^-_{\partial B_r(x_0)}\, d|D\chi_{B_r(x_0)}|.
\]
This map defines, in the metric setting, the interior normal trace of
$\vec F$ on $\partial B_r(x_0)$, and the integration by parts formula
\eqref{IntByParts} holds.
\end{proof}

\begin{remark}
  The term interior normal trace can be justified by the following
  facts. If $\vec F\in {\cal DM}^\infty(\Om)\cap C(\Om,\R^k)$, then by
  Proposition~\ref{dmiCont} we get that
\[
(\vec F\cdot \nu)^-_{\partial B_r(x_0)}=-\scal{\vec F}{\sigma_{x_0,r}}.
\]
In addition, also when $\vec F$ is not continuous, recalling that with $u_{x_0}(x)=d(x,x_0)$,
\[
du_j(x)=-\frac{1}{\eps_j} du_{x_0}(x) \chi_{B_r(x_0)\setminus B_{r-\eps_j}(x_0)}(x), 
\]
and by using the coarea formula \eqref{coarea}, we can write
\begin{align*}
  \int_\Om f\, \scal{\vec F}{du_j}\,d\mu=&
  -\frac{1}{\eps_j}\int_{B_r(x_0)\setminus B_{r-\eps_j}(x_0)} f\, \scal{\vec F}{du_{x_0}}\,d\mu\\
  =& -\frac{1}{\eps_j}\int_{r-\eps_j}^r \int_\Om f\, \scal{\vec F}{du_{x_0}}\, d|D\chi_{B_t(x_0)}|dt.
\end{align*}
Therefore, we have obtained that 
\[
\int_\Om f\, (\vec F\cdot \nu)^-_{\partial B_r(x_0)}\,
d|D\chi_{B_r(x_0)}|= -\lim_{j\to\infty}
\frac{1}{\eps_j}\int_{r-\eps_j}^r\int_\Om f\, \scal{\vec
  F}{du_{x_0}}\, d|D\chi_{B_t(x_0)}|dt,
\]  
which gives meaning to the following equality in terms of the trace
\[
\int_\Om f (\vec F\cdot \nu)^-_{\partial B_r(x_0)}\,d|D\chi_{B_r(x_0)}| = 
-\int_\Om f\scal{\vec F}{du_{x_0}}\, d|D\chi_{B_r(x_0)}|,
\]
and to the fact that the vector $du_{x_0}$ defines in a weak sense the
normal vector $\sigma_{x_0,r}$ to $\partial B_r(x_0)$.
\end{remark}

\begin{remark}
  Observe that in the proof of
  Proposition~\ref{Question1} we have used a particular optimal
  sequence. It turns out, nevertheless, that the interior
  normal trace $(\vec F\cdot \nu)^-_{\partial B_r(x_0)}$ does not
  depend on this particular choice. This fact is a direct consequence
  of equation \eqref{IntByParts}, since then formula
\[
\int_\Om f (\vec F\cdot \nu)^-_{\partial B_r(x_0)}\,
d|D\chi_{B_r(x_0)}| = \int_{B_r(x_0)} f\, d\diver \vec F + \int_{B_r(x_0)}
\scal{\vec F}{df}\, d\mu
\]
uniquely identifies the values of $(\vec F\cdot \nu)^-_{\partial B_r(x_0)}$.
\end{remark}

\begin{remark} \label{rmk:sHausdorff}
  By~\cite[Theorem 5.3]{A} (see also~\cite{AMP}),
  formula~\eqref{IntByParts} can also be written by
\[
\int_{B_r(x_0)} f\, d\diver \vec F + \int_{B_r(x_0)} \scal{\vec F}{df}
\,d\mu=\int_{\partial^*B_r(x_0)} f (\vec F\cdot \nu)^-_{\partial
  B_r(x_0)} \teta_{x_0,r}\,d{\mathcal S}^h,
\]
where $\partial^*B_r(x_0)$ is the essential boundary of $B_r(x_0)$,
${\mathcal S}^h$ is the spherical Hausdorff measure defined using the
Carath\'eodory construction based on the gauge function
\[
h(\overline B_\vrho)=\frac{\mu(\overline B_\vrho)}{\vrho},
\]
and $\teta_{x_0,r}: X\to [c,c_d]$ is a Borel function
depending, in general, on the ball $B_r(x_0)$, and $c$ is a positive constant
and $c_d$ the doubling constant of $\mu$.
\end{remark}

\section{Harmonicity and the mean value property}
\label{sect:SingFunct}

In this section, we shall follow the approach of \cite{HS} and
construct, for any regular ball $B_r(x_0)\subset X$ and any $\bar x\in
B_r(x_0)$ the Green function on $B_r(x_0)$ with singularity at $\bar
x$, that is an extended real-valued function $G(x)=G_{B_r(x_0)}^{\bar
  x}(x)$ such that
\begin{enumerate}
\item $G$ is strictly positive and harmonic in $B_r(x_0)\setminus
 \{\bar x\}$;
\item $G\in N^{1,2}(X\setminus B_\eps(\bar x))$ for any $\eps>0$ and 
$G|_{X\setminus \overline{B}_r(x_0)}=0$;
\item for every $y\in \partial B_r(x_0)$ 
\[
\lim_{x\to y}G(x)=0;
\]
\item $G$ is singular at $\bar x$; that is
\[
\lim_{x\to \bar x} G(x)=\infty;
\]
\item for all $0<a\leq b$,
\[
\Capa_2(\{x\in B_r(x_0):\ G(x)\geq b\},\{x\in B_r(x_0):\ G(x)>a\}) =\frac{1}{b-a}.
\]
\end{enumerate}

In \cite{HS} the authors constructed the Green function of a
relatively compact domain with the aforementioned properties in metric
measure spaces; we refer also to \cite{H} and \cite{DGM}. We can state
the existence and main properties of the Green function in the
following theorem. We assume that $X$ supports a $(1,2)$-Poincar\'e
inequality.

\begin{theorem}\label{thmGreen} 
  Let $\Omega\subset X$ be a relatively compact domain. Then there
  exists the Green function $G=G_\Omega^{\bar x}$ with singularity at
  $\bar x\in \Omega$. In addition, $dG\in L^2 (X\setminus B_\eps(\bar
  x))$ for any $\eps>0$ and 
\[
\mathscr{D}_{X\setminus B_\eps(\bar x)} G=-\nu^G_\Omega,
\]
where $\nu_\Om^G$ is a positive Radon measure in the dual
$N^{1,2}_0(X\setminus B_\eps(\bar x))^*$ concentrated on $\partial
\Om$. Moreover, $G$ admits the measure-valued Laplace operator
\[
\mathscr{D}_XG=\delta_{\bar x}-\nu^G_\Om,
\]
in the sense that for any $v\in N^{1,2}(X)$ continuous at $\bar x$,
then
\[
\int_X \scal{dG}{dv}\,d\mu=\int_{\partial \Om} v\,d\nu^G_\Om-v(\bar x).
\]
\end{theorem}

\begin{proof}
We refer to \cite{HS} for the details on the construction of $G$. We sketch the main
steps needed in the definition. We find a harmonic function
on $\Om\setminus \overline B_{\eps_j}(\bar x)$ 
\[
G_j =\frac{v_j}{\Capa_2(\overline{B}_{\eps_j}(\bar x),\Om)},
\]
where $B_{\eps_j}(\bar x)$ is a regular ball, $\eps_j\searrow 0$,
$\eps_j<{\rm dist}(\bar x,\partial \Om)$, and $v_j$ is the potential
of $\overline{B}_{\eps_j}(\bar x)$ with respect to $\Om$; that is
$v_j\in N^{1,2}(X)$ is harmonic in
$\Om\setminus\overline{B}_{\eps_j}(\bar x)$, $v_j=0$ on $X\setminus
\Om$ and $v_j=1$ on $\overline B_{\eps_j}(\bar x)$. It is then shown
that, up to subsequences, the functions $(G_j)_j$ converge locally
uniformly in $X\setminus \{\bar x\}$ to a function $G$. The limit
function $G$ has the desired properties of a Green function.

Let us fix a positive sequence $(M_i)_{i\geq 0}$ such that
$M_i\nearrow\infty$, and the truncations
\[
T_iG:=\min\{G,M_i\}.
\]
There exists a sequence $r_i\searrow 0$ of radii such that
\[
E_i \subset B_{r_i}(\bar x),
\]
where we have written $E_i=\{x\in \Om: G(x)>M_i\}$; and we may consider the case in which $r_i<\eps$. Then $T_iG=G$ on $X\setminus B_\eps(\bar x)$ and $T_iG$ is subharmonic
in $X\setminus B_\eps(\bar x)$. By \cite{BMS} (we refer also to
\cite{M} for a detailed description in the Euclidean case) there
exists a positive Radon measure $\nu^G_\Om$ in the dual $N^{1,2}_0(X\setminus
B_\eps(\bar x))^*$ such that for all $v\in \Lip_c(X\setminus
B_\eps(\bar x))$ we have
\[
\int_{X\setminus B_\eps(\bar x)} \scal{dG}{dv}\,d\mu=\int_{X\setminus
  B_\eps(\bar x)} v\,d\nu^G_\Om.
\]
If $v\in {\rm Lip}_c(X\setminus \overline\Omega)$, 
the fact that $G=0$ on $X\setminus\overline \Omega$ implies $dG=0$ on $X\setminus \overline \Omega$, and then
\[
\int_{X\setminus B_\eps (\bar x)} v\,d\nu^G_\Omega =
\int_{X\setminus \overline \Omega} v\,d\nu^G_\Omega=\int_{X\setminus \overline\Omega} 
\scal{dG}{dv}\,d\mu =0.
\]
On the other hand, the harmonicity of $G$ in $\Omega\setminus B_\eps(\bar x)$ implies
that if $v\in {\rm Lip}_c(\Omega\setminus B_\eps(\bar x))$, then
\[
\int_{X\setminus B_\eps (\bar x)} v\,d\nu^G_\Omega =
\int_{\Omega\setminus B_\eps (\bar x)} v\,d\nu^G_\Omega =
\int_{\Omega\setminus B_\eps(\bar x)} 
\scal{dG}{dv}\,d\mu =0.
\]
Hence the measure $\nu^G_\Omega$ is concentrated on $\partial \Om$.

Analogously, since $T_iG$ is superharmonic in $\Om$ there exists a
positive Radon measure $\nu^G_i\in N^{1,2}_0(\Om)^*$ such that for all $v\in
\Lip_c(\Om)$
\[
\int_\Om \scal{dT_iG}{dv}\,d\mu=-\int_{\Om\setminus B_{r_i}(\bar x)} v\,d\nu^G_i.
\]
The measures $\nu_i$ are supported in $\overline{B}_{r_i}(\bar x)$; indeed, since $T_iG=G$ on $\Om\setminus \overline{B}_{r_i}(\bar x)$ it is harmonic. Hence, if $v\in {\rm Lip}_c(\Om\setminus \overline{B}_{r_i}(\bar x))$, 
\[
\int_\Om v\,d\nu^G_i 
=\int_{\Om \setminus \overline{B}_{r_i}(\bar x)} v\,d\nu^G_i =\int_{\Om\setminus \overline{B}_{r_i}(\bar x)}\scal{dG}{dv}\,d\mu=0. 
\]
Following the argument of Serrin~\cite[Lemma 1 and Theorem 3]{Se}, 
there exists $\lambda\in \R$ such that if $v\in {\rm Lip}_c(\Om)$ is equal to $1$ 
in a neighborhood of $\bar x$, then
\[
\int_\Om \scal{dG}{dv}\,d\mu=\lambda.
\]
Indeed, if $v_1,v_2\in {\rm Lip}_c(\Om)$ are two functions that are equal to $1$ in a neighborhood
of $\bar x$, the difference $v=v_1-v_2$ belongs to ${\rm Lip}_c(\Om\setminus\{\bar x\})$; hence, the harmonicity
of $G$ in $\Om\setminus\{\bar x\}$ implies that
\[
\int_\Om \scal{dG}{dv_1}\,d\mu -\int_\Om \scal{dG}{dv_2}\,d\mu=\int_\Om \scal{dG}{dv}\,d\mu=0.
\]
In particular, if $v\in {\rm Lip}_c(\Om)$ is a function such that $v\equiv 1$
on $\overline{B}_{r_1}(\bar x)$, then 
\[
\nu^G_i(\overline{B}_{r_i}(\bar x))=\int_\Om v\,d\nu^G_i 
=-\int_\Om \scal{dT_iG}{dv}\,d\mu
=-\int_\Om \scal{dG}{dv}\,d\mu=-\lambda.
\]
This argument implies that $\lambda\in\R$ is negative and the measures $\nu^G_i$ are equibounded in $\mathscr{M}_b(\Omega)$.
Thus, up to subsequences, $\nu_i^G$ converges weakly to $\lambda\delta_{\bar x}$.
To summarize, we have proved that the sequence of the measure-valued Laplace
operators
\[
\mathscr{D} T_iG=\nu_i^G-\nu^G_\Om
\]
admits a convergent subsequence $\mathscr{D}T_{i_k} G$,
defining the measure-valued Laplace operator
\[
\mathscr{D}_X G=\lim_{k\to\infty} \mathscr{D}T_{i_k} G= \lambda
\delta_{\bar x}-\nu^G_\Om.
\] 
The fact that the limit measure is uniquely determined implies that
for any sequence $M_i\nearrow\infty$, the measures $\mathscr{D}T_iG$
converge and the limit measure is $\lambda \delta_{\bar x}-\nu^G_\Omega$.

Let us show that $\lambda=-1$. 
Let us consider the set $E=\{x\in \Om:\
G(x)\geq 1\}$ and a function $v\in \Lip_c(\Om)$ such that $v=1$ on
$E$. Since $\bar x$ is an interior point of $E$, we have
\[
\lambda =\lambda v(\bar x)=-\int_{\Om\setminus E} \scal{dv}{dG}\,d\mu.
\]
On the other hand, the map $f=(G-v)\chi_{\Om\setminus E}$ belongs to
$N^{1,2}_0(\Om\setminus E)$ and then
\[
0= \int_\Om \scal{df}{dG}\,d\mu = \int_{\Om\setminus E} |dG|^2\,d\mu
-\int_{\Om\setminus E} \scal{dv}{dG}\,d\mu.
\]
These properties of $G$ imply that $G$ is the potential of $E$ with
respect to $\Om$, that is
\[
\int_{\Om\setminus E} |dG|^2\,d\mu ={\rm Cap}_2(E,\Om)=1.
\]
We may hence conclude that $\lambda=-1$.

Finally, we point out that the identity
\[
\int_X \scal{dv}{dG}\,d\mu=\int_{\partial \Omega} v\,d\nu^G_\Om - v(\bar x)
\]
is valid for functions $v\in N^{1,2}(X)$ that are constant in a neighborhood
of $\bar x$, but it can be generalized to functions $v\in N^{1,2}(X)$ that
are continuous at $\bar x$. This is a simple consequence of the
limit
\begin{align*}
\int_{\partial \Om} v\,d\nu^G_\Om -v(\bar x)=&
\lim_{i\to \infty} \int_{\partial \Om} v\,d\nu^G_\Om -\int_{B_{r_i}(\bar x)}v\,d\nu^G_i\\
=&
\lim_{i\to \infty} \int_X \scal{dv}{dT_i G}\,d\mu =\int_X \scal{dv}{dG}\,d\mu.
\end{align*}
\end{proof}

\begin{remark}
Let us consider the (first) Heisenberg group $\mathbb{H}$ with the geodesic 
distance. In this case, the natural differential structure is given by the horizontal bundle and the Laplace operator is just the horizontal Laplace operator. In this setting,
we can use all the results of the preceding section and obtain the representation
of the measure $\nu^G$ in terms of the perimeter measure. Notice that 
a ball $B_r(x_0)$ in $\mathbb{H}$ satisfies a ball condition as in \cite[Definition 2.1]{AKSZ} at its boundary except at two points; a finite collection of points is
negligible. Hence, if $G=G^{\bar x}_{B_r(x_0)}$ is the Green function on $B_r(x_0)$ with singularity at $\bar x$, then whenever $x$ is a boundary point of  $B_r(x_0)$ 
satisfying the ball condition,
\begin{equation} \label{eq:twosidedballcond}
\Psi(G,x,\varrho):= \sup_{B_{2\varrho}(x)} G -\sup_{B_\varrho(x)}G
\leq C\varrho,
\end{equation}
where $0<\rho\leq d(x,\bar x)/2$ and $C$ is a positive constant that does not depend on $x$, $\bar x$, or $\rho$. It follows from a covering argument together with \eqref{eq:twosidedballcond} and \cite[Lemma 4.8]{BMS} that $\nu^G_{B_r(x_0)}$ is absolutely continuous with respect to the perimeter measure $|D\chi_{B_r(x_0)}|$. Moreover,               
there exists a function $\vartheta_G\in L^1(X,|D\chi_B|)$ such that
 $d\nu_B^G=\vartheta_G d|D\chi_B|$. The function $\vartheta_G$ comes
 from the Radon--Nikodym theorem. 
\end{remark}

We give a characterization of harmonic functions via a mean value type
property with respect to boundary measures.

\begin{theorem}\label{meanValThm}
Let $u\in N^{1,2}(\Om,\mu)$, then the following hold:
\begin{enumerate}
\item[(1)] Let $u$ be harmonic in $\Om$. Then for every regular ball
  $B_r(x_0)\subset\Om$ and $\bar x\in B_r(x_0)$
\begin{equation}\label{meanValue}
u(\bar x)=\int_{\partial B_r(x_0)} u\, d\nu^G_{B_r(x_0)};
\end{equation}
\item[(2)] If for every regular ball $B_r(x_0)\subset \Om$ and any
  $\bar x\in B_r(x_0)$, $u$ satisfies the mean value property
  \eqref{meanValue}, then $u$ is harmonic in $\Om$.
\end{enumerate}
An analogous characterization holds true for sub- and superharmonic
functions. Let $u\in N^{1,2}(\Om,\mu)$ then the following are
equivalent:
\begin{enumerate}
\item[(3)] Let $u$ be subharmonic (superharmonic) in $\Om$. Then for
  every regular ball $B_r(x_0)\subset \Om$ and $\bar x\in B$
\[
u(\bar x)\leq \int_{\partial B_r(x_0)} u\, d\nu^G_{B_r(x_0)}, \qquad 
\left(u(\bar x)\geq \int_{\partial B_r(x_0)} u\, d\nu^G_{B_r(x_0)}\right);
\]
\item[(4)] If for any regular ball $B_r(x_0)$ and any $\bar x\in
  B_r(x_0)$ 
  \[
  u(\bar x)\leq \int_{\partial B_r(x_0)} u\, d\nu^G_{B_r(x_0)}, \qquad \left(u(\bar
    x)\geq \int_{\partial B_r(x_0)} u\, d\nu^G_{B_r(x_0)}\right),
\]
then $u$ is subharmonic (superharmonic).
\end{enumerate}
\end{theorem}

\begin{proof}
  Suppose that $u$ is harmonic. Then $u\in N^{1,2}(\Om,\mu)\cap
  L^{\infty}_{\rm loc}(\Om)$ and we can apply
  Theorem~\ref{thmGreen}. We obtain for any regular ball $B_r(x_0)$
  and $\bar x\in B_r(x_0)$
\[
0=\int_X \scal{du}{dG^{\bar x}_B}\,d\mu =-\int_X u\,d\mathscr{D}_X G^{\bar
  x}_B =-u(\bar x)+\int_{\partial B_r(x_0)} u\, d\nu^G_B,
\]
which gives the condition (1).

On the other hand, if $u$ is continuous, if we fix a regular ball $B=B_r(x_0)$,
we can consider the harmonic function $H_u$ generated by $u$ on $B$, that is the
solution of the problem
\[
\min\left\{
\int_B |dv|^2 \,d\mu: v-u\in N^{1,2}_0(B,\mu)\right\}.
\]
Hence $H_u$ is harmonic in $B$ and satisfies the mean value property, that is
for any $\bar x\in B$
\begin{equation}\label{huMean}
H_u(\bar x)=\int_{\partial B_r(x_0)} H_u d\nu^G_B.
\end{equation}
The conclusion follows from continuity of $u$ since
\[
\lim_{B\ni x\to y\in \partial B}H_u(x)=u(y),
\]
and then by \eqref{huMean}, $H_u=u$ on $B$. For a general $u\in
N^{1,2}(\Om,\mu)$, we can find a continuous function $u_\eps$ such
that $u=u_\eps$ outside a set of capacity less than $\eps$ and such
that $\|u-u_\eps\|_{1,2}<\eps$; then by an approximation
argument in \cite[Section~6]{BBS}, we can conclude the assertion.

The same line of reasoning carries out in the case of sub- and
superharmonic functions.
\end{proof}

\begin{remark}
  It was proved in \cite{BBS} that the harmonic extension of a
  function $u\in N^{1,2}(\Om,\mu)$ on a ball $B\subset \Om$ can be
  expressed in terms of harmonic measures $\nu_{\bar x}$ with
  singularity at $\bar x\in B$; by this we mean
  that if $\varphi\in C(\partial B)$, then in \cite[Theorem 5.1]{BBS}, 
  its harmonic extension is given by
  \[
  H_\varphi(\bar x)=\int_{\partial B}\varphi\, d\nu_{\bar x}.
  \] 
  If we move $\bar x\in B_r(x_0)$, it is possible to see that the measures
  $\nu_{\bar x}$ are mutually equivalent; in particular, if we take $x_0$ and $\bar x\in B_r(x_0)\setminus\{x_0\}$, we have that $\nu_{\bar x}$ is absolutely continuous
  with respect to $\nu_{x_0}$ and its density $P(\bar x,\cdot)$ is called
  the Poisson kernel. In other terms, the Poisson kernel is defined as
  \[
  P(\bar x,x)=\frac{d\nu_{\bar x}}{d\nu_{x_0}}(x).
  \]
  In \cite{BBS}, $\nu_{\bar x}$ was not explicitly identified. 
 Nevertheless, from the results contained in the previous sections, we are able to
 identify this measure as the outward normal derivative 
 $\nu_B^G$ of the Green function.
\end{remark}

\begin{example}
{\rm 
In Example~\ref{example1}, if we take $\Omega=B_1(0)$, the unit ball, 
then all balls except $B_1(0)$ are regular. This is
due to the fact that the perimeter of $B_1(0)$ has weight $1$, that
is $|D\chi_B|={\cal H}^{n-1}\res \partial B$. However, if we
consider the optimal sequence $(u_j)_j$ defined in \eqref{optimalseq} we have
that
\[
\int_\Rn |\nabla u_j|\, d\mu \to 2 {\cal H}^{n-1}(\partial
B_1(0))=2|D\chi_{B_1(0)}|(\Rn).
\]
Nevertheless, the measure $\nu^G_{B_1(0)}$ can still be characterized
as a perimeter measure, but with
\[
d\nu^G_{B_1(0)}= 2(\nabla G \cdot \nu_{B_1(0)}) d{\cal H}^{n-1}\res \partial B_1(0)
=2 (\nabla G \cdot \nu_{B_1(0)}) d|D\chi_{B_1(0)}|.
\]
On the other hand, if we take any other ball $B\subset B_1(0)$, it is
regular and in this case ${\cal H}^{n-1}(\partial B\cap \partial
B_1(0))=0$. Note also that if ${\cal H}^{n-1}(B\cap \partial
B_1(0))>0$, then since the Green function is harmonic in $B$ except the
singular point $\bar x$, we have that $\nabla G\cdot \nu_{B_1(0)}=0$
and then
\[
d\nu^G_B = (1+\chi_{B_1(0)})(\nabla G \cdot \nu_B) d{\cal H}^{n-1}\res \partial B
=(\nabla G \cdot \nu_B) d|D\chi_B|
\]
On the other hand, if we take $\Om=\Rn\setminus \overline{B_1}(0)$,
then every ball is regular. This is due to the fact that in this paper
regularity is a notion of inner regularity. If one changes the notion
to outer regularity or to two-sided regularity, then things change.}
\end{example}

\noindent Addresses:\\

\noindent N.M.: Department of Mathematics and Statistics, University of Helsinki, \\
P.O. Box 68 (Gustaf H\"allstr\"omin katu 2b), FI-00014 University of Helsinki, Finland. \\
\noindent E-mail: {\tt niko.marola@helsinki.fi} \\

\noindent M.M.: Department of Mathematics and Computer Science, University of Ferrara, \\ 
via Machiavelli 35, 44121, Ferrara, Italy. \\
\noindent E-mail: {\tt michele.miranda@unife.it} \\

\noindent N.S.: Department of Mathematical Sciences, University of Cincinnati, \\
P.O.Box 210025, Cincinnati, OH 45221--0025, USA. \\
\noindent E-mail: {\tt shanmun@uc.edu} \\

\end{document}